\documentclass[11pt, reqno]{amsart}

\title{Unipotent Representations and the Categorical Trace of Frobenius}
\author{Wyatt Reeves}
\date{October 22, 2026}

\usepackage[utf8]{inputenc}
\usepackage[T1]{fontenc}
\usepackage{baskervald}
\usepackage{amsmath}
\usepackage{amssymb}
\usepackage{amsfonts}
\usepackage{amsthm}
\usepackage[all,arc]{xy}
\usepackage{enumerate}
\usepackage{mathrsfs}
\usepackage[dvipsnames]{xcolor}
\usepackage{appendix}
\usepackage{changepage}
\usepackage{slashed}
\usepackage{tikz-cd}
\usepackage{tcolorbox}
\usepackage{hyperref}
\usepackage{cleveref}
\usepackage{thmtools,thm-restate}
\usepackage{verbatim}
\usepackage{stmaryrd}
\usepackage{spectralsequences}
\usepackage{mathtools}
\usepackage{dynkin-diagrams}
\usepackage[mathscr]{euscript}

\usepackage{graphicx}
\graphicspath{ {images/} }

\DeclareMathSymbol{A}{\mathalpha}{operators}{`A}
\DeclareMathSymbol{B}{\mathalpha}{operators}{`B}
\DeclareMathSymbol{C}{\mathalpha}{operators}{`C}
\DeclareMathSymbol{D}{\mathalpha}{operators}{`D}
\DeclareMathSymbol{E}{\mathalpha}{operators}{`E}
\DeclareMathSymbol{F}{\mathalpha}{operators}{`F}
\DeclareMathSymbol{G}{\mathalpha}{operators}{`G}
\DeclareMathSymbol{H}{\mathalpha}{operators}{`H}
\DeclareMathSymbol{I}{\mathalpha}{operators}{`I}
\DeclareMathSymbol{J}{\mathalpha}{operators}{`J}
\DeclareMathSymbol{K}{\mathalpha}{operators}{`K}
\DeclareMathSymbol{L}{\mathalpha}{operators}{`L}
\DeclareMathSymbol{M}{\mathalpha}{operators}{`M}
\DeclareMathSymbol{N}{\mathalpha}{operators}{`N}
\DeclareMathSymbol{O}{\mathalpha}{operators}{`O}
\DeclareMathSymbol{P}{\mathalpha}{operators}{`P}
\DeclareMathSymbol{Q}{\mathalpha}{operators}{`Q}
\DeclareMathSymbol{R}{\mathalpha}{operators}{`R}
\DeclareMathSymbol{S}{\mathalpha}{operators}{`S}
\DeclareMathSymbol{T}{\mathalpha}{operators}{`T}
\DeclareMathSymbol{U}{\mathalpha}{operators}{`U}
\DeclareMathSymbol{V}{\mathalpha}{operators}{`V}
\DeclareMathSymbol{W}{\mathalpha}{operators}{`W}
\DeclareMathSymbol{X}{\mathalpha}{operators}{`X}
\DeclareMathSymbol{Y}{\mathalpha}{operators}{`Y}
\DeclareMathSymbol{Z}{\mathalpha}{operators}{`Z}

\newcommand{\sA}{\mathscr{A}}

\newcommand{\sB}{\mathscr{B}}

\newcommand{\sC}{\mathscr{C}}

\newcommand{\sD}{\mathscr{D}}

\newcommand{\sF}{\mathscr{F}}

\newcommand{\sG}{\mathscr{G}}

\newcommand{\sH}{\mathscr{H}}

\newcommand{\sJ}{\mathscr{J}}

\newcommand{\sL}{\mathscr{L}}

\newcommand{\sM}{\mathscr{M}}

\newcommand{\sO}{\mathscr{O}}

\newcommand{\sW}{\mathscr{W}}

\newcommand{\sY}{\mathscr{Y}}

\newcommand{\A}{\mathbf{A}}

\newcommand{\E}{\mathbf{E}}
\newcommand{\F}{\mathbf{F}}

\newcommand{\bO}{\mathbf{O}}
\newcommand{\bP}{\mathbf{P}}
\newcommand{\Q}{\mathbf{Q}}

\newcommand{\V}{\mathbf{V}}
\newcommand{\W}{\mathbf{W}}

\newcommand{\Z}{\mathbf{Z}}

\DeclareFontFamily{U}{min}{}
\DeclareFontShape{U}{min}{m}{n}{<-> udmj30}{}

\newcommand{\set}[1]{\left\lbrace #1 \right\rbrace}

\newcommand{\bracket}[1]{\left\langle #1 \right\rangle}

\newcommand{\p}{\partial}
\newcommand{\pr}[1]{\frac{\partial}{\partial #1}}

\newcommand{\inv}{^{-1}}

\newcommand{\ol}[1]{\overline{#1}}

\newcommand{\totext}[1]{\xrightarrow{#1}}

\newcommand{\mto}{\hookrightarrow}

\newcommand{\ch}[1]{\check{#1}}
\newcommand{\bs}{\backslash}
\newcommand{\sbullet}{{\mathbin{\vcenter{\hbox{\scalebox{.5}{$\bullet$}}}}}}
\newcommand{\op}{\mathrm{op}}

\newcommand{\dgCat}{\mathrm{dgCat}}

\newcommand{\oblv}{\mathrm{oblv}}
\newcommand{\ind}{\mathrm{ind}}

\newcommand{\QLisse}{\mathrm{QLisse}}

\newcommand{\tmod}{\textrm{-mod}}

\newcommand{\ins}{\mathrm{ins}}

\newcommand{\Ad}{\mathrm{Ad}}

\newcommand{\unip}{\mathrm{unip}}

\newcommand{\Shv}{\mathrm{Shv}}

\newcommand{\rev}{\mathrm{rev}}
\newcommand{\Gm}{{\mathbf{G}_m}}
\newcommand{\bimod}{\mathrm{-bimod}}
\newcommand{\mult}{\mathrm{mult}}
\newcommand{\unit}{\mathrm{unit}}
\newcommand{\triv}{\mathrm{triv}}

\newcommand{\ren}{\mathrm{ren}}

\newcommand{\act}{\mathrm{act}}
\newcommand{\Weil}{\mathrm{Weil}}
\newcommand{\Fr}{\mathrm{Fr}}
\newcommand{\coact}{\mathrm{coact}}
\newcommand{\relev}{\mathrm{relev}}
\newcommand{\Perv}{\mathrm{Perv}}
\newcommand{\access}{\mathrm{access}}
\newcommand{\inaccess}{\mathrm{inaccess}}
\newcommand{\Alg}{\mathrm{Alg}}
\newcommand{\Fil}{\mathrm{Fil}}
\newcommand{\Spr}{\mathrm{Spr}}
\newcommand{\naive}{\mathrm{naive}}
\newcommand{\Frob}{\mathrm{Frob}}
\newcommand{\Av}{\mathrm{Av}}

\DeclareMathOperator{\Hom}{Hom}
\DeclareMathOperator*{\colim}{colim}

\DeclareMathOperator{\End}{End}

\DeclareMathOperator{\tr}{Tr}

\DeclareMathOperator{\coker}{coker}
\DeclareMathOperator{\Spec}{Spec}

\DeclareMathOperator{\gr}{gr}

\DeclareMathOperator{\Vect}{Vect}

\DeclareMathOperator{\id}{id}

\DeclareMathOperator{\Ind}{Ind}
\DeclareMathOperator{\Perf}{Perf}

\DeclareMathOperator{\QCoh}{QCoh}

\DeclareMathOperator{\ICoh}{IndCoh}

\DeclareMathOperator{\cof}{cof}
\DeclareMathOperator{\Irr}{Irr}
\DeclareMathOperator{\Conj}{Conj}

\DeclareMathOperator{\Rep}{Rep}

\DeclareMathOperator{\ev}{ev}

\DeclareMathOperator{\Sh}{Sh}

\renewcommand{\thekevinthm}{\arabic{section}.\arabic{subsection}.\arabic{subsubsection}}
\newcommand\newkevintheorem[3][]{%
    \newtheorem{#2}[kevinthm]{#3}%
    \csletcs{old#2}{#2}%
    \expandafter\renewcommand\csname #2\endcsname{\stepcounter{subsubsection}\csname old#2\endcsname}%
    \expandafter\def\csname theH#2\endcsname{\thekevinthm}
}

\crefname{appsec}{Appendix}{Appendices}

\newtheoremstyle{kevintheoremstyle}
{}
{}
{}
{0pt}
{\bfseries}
{.}
{4pt}
{\thmnumber{#2}\thmname{ #1}\thmnote{ (#3)}}

\theoremstyle{kevintheoremstyle}

\newkevintheorem{theorem}{Theorem}
\crefname{theorem}{Theorem}{Theorems}
\newkevintheorem{lemma}{Lemma}
\crefname{lemma}{Lemma}{Lemmas}
\newkevintheorem{proposition}{Proposition}
\crefname{proposition}{Proposition}{Propositions}
\newkevintheorem{corollary}{Corollary}
\crefname{corollary}{Corollary}{Corollaries}
\newkevintheorem{conjecture}{Conjecture}
\crefname{conjecture}{Conjecture}{Conjectures}
\newkevintheorem{construction}{Construction}
\crefname{construction}{Construction}{Constructions}
\newkevintheorem{computation}{Computation}
\crefname{computation}{Computation}{Computation}

\hypersetup{
    colorlinks=true,
    linkcolor=Maroon,
    citecolor=MidnightBlue,
    urlcolor=MidnightBlue
}

\makeatletter
\def\l@subsection{\@tocline{2}{0pt}{2.5pc}{5pc}{}}
\makeatother

\begin{document}

\begin{abstract}
We give a new proof of the classification of unipotent representations for split finite groups of Lie type. To do this, we first construct an $(\infty, 2)$-categorification of Lusztig's map from the Hecke algebra to the asymptotic algebra. Then we use the theory of weights to show that it induces an isomorphism after applying trace of Frobenius. 
\end{abstract}

\maketitle

\tableofcontents


\section{Introduction}

\subsection{Motiviation}

\subsubsection{} A basic problem in representation theory is to understand the category of representations of a group. When $\Gamma$ is a finite group and the field of coefficients has characteristic zero, a classical theorem of H. Maschke says that $\Rep(\Gamma)$ is semisimple. As a result, to describe the structure of $\Rep(\Gamma)$ in this case, it suffices to give a parameterization of the irreducible representations of $\Gamma$. 

\subsubsection{} In a sense, the most basic case where one can try to enumerate the irreducible representations of $\Gamma$ is when the group is simple and the field of coefficients is algebraically closed. If $\Gamma$ is $\Z/p$ then the answer is very simple, and if $\Gamma$ is one of the sporadic simple groups the irreducible representations can be enumerated by hand. If $\Gamma$ is the alternating group $A_n$, then because of the short exact sequence 
\[1 \to A_n \to S_n \to \Z/2 \to 1,\]
and the vanishing of $H^2(\Z/2; \Gm)$, it is relatively straightforward to describe the irreducible representations of $A_n$ in terms of those of $S_n$. The classical work of G. Frobenius, A. Young, I. Schur, and A. Specht then provides a parameterization of the irreducible representations of the symmetric group $S_n$ in terms of partitions of $n$.

\medskip 

The remaining finite simple groups are all of \textit{Lie type}, meaning that there is a short exact sequence 
\[1 \to \Gamma \to G^F \to A \to 1\]
where $G$ is a simple algebraic group of adjoint type defined over $\ol{\F}_p$, $F: G \to G$ is a Steinberg endomorphism of $G$ (meaning that some power of $F$ is a Frobenius endomorphism of $G$), and $A$ is a finite abelian group. Like in the case of the alternating group, the key input needed to understand the representations of $\Gamma$ is an understanding of the representations of $G^F$. For simplicity, we will restrict our attention to the case that $F$ is actually a Frobenius automorphism of $G$: in this case the group $G^F$ identifies with $G(\F_q)$. 

\subsubsection{} In the work of P. Deligne and G. Lusztig \cite{deligne1976representations}, it was shown that the irreducible representations of $G(\F_q)$ fall into families parameterized by semisimple elements in $\ch{G}(\F_q)$ up to geometric conjugacy, where $\ch{G}$ denotes the Langlands dual group. The representations whose semsisimple parameter is $1 \in \ch G(\F_q)$ are called \textit{unipotent}. For a general semisimple element $s \in \ch{G}(\F_q)$, a theorem of R. Steinberg and T. Springer \cite[Lemma 3.9.(2)]{steinberg1963representations} shows that the centralizer of $s$ is a connected reductive group. Furthermore, by the work of Lusztig \cite{lusztig1984characters} (see also \cite{LUSZTIG_YUN_2020}) the category of representations of $G(\F_q)$ with semisimple parameter $s$ is naturally equivalent to the category of unipotent representations of $H(\F_q)$, where $H$ is the Langlands dual group of $Z(s)$ (although $Z(s)$ may not be split even if $G$ is).     

\subsubsection{}\label{sss:lusztig-unipotent-reps} In \cite{lusztig1984characters}, Lusztig also showed that the irreducible unipotent representations of $G(\F_q)$ themselves break into families parameterized by the two-sided cells of the Weyl group of $G$. Furthermore, to each two-sided cell $c$, one can attach a finite group $A_c$ such that the irreducible unipotent representations in the family associated to $c$ are parameterized by pairs $(x, \sigma)$, where $x$ is an element of $A_c$ and $\sigma$ is an irreducible representation of the centralizer of $x$ in $A_c$. 

\medskip 

Despite the fact that Lusztig's result can be stated uniformly in the group $G$, the proof relies on extensive type-by-type calculations. This has the advantage of being very explicit, but it is natural to wonder whether there is a uniform proof of the classification of unipotent representations. The goal of this paper is to provide a uniform proof of a result which is slightly weaker than the one stated above. Namely, we will prove 

\begin{theorem}[\Cref{thm:main}]\label{thm:intro-main}
The category $\Rep(G(\F_q))^\unip$ admits a direct sum decomposition indexed by two-sided cells 
\[\Rep(G(\F_q))^\unip \simeq \bigoplus_c \Rep(G(\F_q))^\unip_c\]
such that 
\[\Rep(G(\F_q))^\unip_c \simeq \QCoh_\sG(A_c/A_c)\]
for a finite group $A_c$ and gerbe $\sG$.  
\end{theorem}

As we have stated it, Lusztig's classification follows from this theorem plus the knowledge that the gerbe $\sG$ is always trivializable. Using the prior work of \cite{Losev_Ostrik_2014}, it is easy to verify that in all cases the gerbe $\sG$ is trivializable; unfortunately we cannot give a uniform proof of this fact. On the other hand, we should emphasize that even though $\sG$ is always trivializable, it may not have a canonical trivialization. This causes the appearance of certain constant factors in the formulas for the characters of unipotent representations in types $E_7$ and $E_8$ (see \cite{ostrik2014tensor}). 

\medskip

We should caution that although we certainly expect that our partition of $\Rep(G(\F_q))^\unip$ into families agrees with that of \cite{lusztig1984characters}, we do not prove this. Additionally, our strategy for proving \Cref{thm:intro-main} is to first prove an equivalence 
\begin{equation}\label{eq:intro-unip-Jc}
\Rep(G(\F_q))^\unip_c \simeq HH(\sJ_c)    
\end{equation}
(using our definition of $\Rep(G(\F_q))^\unip_c$) and then apply the results of \cite{Losev_Ostrik_2014}. In \cite{lusztig2015unipotent}, Lusztig proves a statement equivalent to (\ref{eq:intro-unip-Jc}) (using his definition of $\Rep(G(\F_q))^\unip_c$), so in a sense the contribution of this paper is simply to give a new proof of this result. See \S\ref{sss:lusztig-paper} for more discussion.

\subsection{Overview}

\subsubsection{}\label{sss:intro-hh-defn} The goal of this section is to give an overview of our proof of \Cref{thm:intro-main}. Our point of departure is a recent result of A. Eteve \cite{eteve2024free}. To state it, let's introduce some notation: if $\sA$ is a monoidal dg category and $\sM$ is an $(\sA, \sA)$-bimodule category, we will write $HH(\sA; \sM)$ for the category  
\[\sA \underset{\sA \otimes \sA^\rev}{\otimes} \sM.\] 
As a special case, if $\phi$ is a monoidal automorphism of $\sA$, then we will write
\[HH(\sA; \phi) = HH(\sA; \sA_\phi)\]
where $\sA_\phi$ is the $(\sA, \sA)$-bimodule category where the right action of $\sA$ has been twisted by $\phi$. 

\medskip 

If $G$ is a reductive group with Borel subgroup $B$, let $\sH$ denote the bi-equivariant Hecke category
\[\sH = \Sh(B \bs G /B)\]
where $\Sh(-)$ denotes the dg category of ind-constructible $\ol{\Q}_\ell$-sheaves. $\sH$ is a monoidal category via convolution, and Eteve's theorem shows that 
\[HH(\sH; \Frob) \simeq \Rep(G(\F_q))^\unip.\]
(One caveat: this result doesn't appear as stated in \cite{eteve2024free}, but it follows from \cite[Theorem 6.1.1]{eteve2024free} by setting the semisimple parameter to 1. Alternatively, one can deduce this equivalence from an even more recent and very general result of X. Zhu \cite[Proposition 8.57]{zhu2025tamecategoricallocallanglands}, see \S\ref{sss:trace-of-hecke} for more discussion.)

\subsubsection{} Let us now explain how the decomposition of unipotent representations into families indexed by two-sided cells arises from this point of view. Recall from \cite[Definition 1.4]{benzvi2015charactertheorycomplexgroup}, \cite[Appendix C]{arinkin2022stacklocalsystemsrestricted} that a monoidal dg category $\sA$ is said to be \textit{semi-rigid} if 
\begin{enumerate}
    \item The multiplication map $\mult: \sA \otimes \sA \to \sA$ admits a continuous $(\sA, \sA)$-bilinear right adjoint.
    \item $\sA$ is dualizable.
\end{enumerate}
In \Cref{constr:semiorthogonal-trace}, we show that in a high degree of generality, semiorthogonal decompositions of semi-rigid monoidal categories induce semiorthogonal decompositions after passing to $HH(-)$. Because $\Rep(G(\F_q))^\unip$ is semi-simple, every semiorthogonal decomoposition of $\Rep(G(\F_q))^\unip$ is actually a direct sum decomposition. 

In \cite[Theorem 1.8]{benzvi2015charactertheorycomplexgroup}, D. Ben-Zvi and D. Nadler prove that the monoidal category $\sH$ is semi-rigid, so to produce a direct sum decomposition of $\Rep(G(\F_q))^\unip$ indexed by two sided cells, it suffices to produce a system of semiorthogonal decompositions of $\sH$ indexed by two-sided cells. 

\subsubsection{}\label{sss:intro-cells} Let us now briefly recall the definition of two-sided cells. First, one defines a preorder on the Weyl group as follows: for any $w \in W$, we can consider the full subcategory $\sH_w^\access \subseteq \sH$ which is the smallest cocomplete two-sided ideal containing $L_w$; say $v \le_{LR} w$ if $L_v$ is contained in $\sH_w$. Using this preorder we may partition $W$ into subsets called two-sided cells by stipulating that $w$ and $v$ lie in the same two-sided cell when $v \le_{LR} w$ and $w \le_{LR} v$. One can define left cells and right cells similarly, by replacing $\sH_w$ with the smallest cocomplete left ideal (resp. right ideal) containing $L_w$. 

\medskip 

If $c$ is a two sided cell, then we will write $w \le_{LR} c$ to mean $w \le_{LR} v$ for some (equiv. any) $v \in c$. We will also write $\sH_{\le c}^\access$ to refer to the category we called $\sH_w^\access$ for $w \in c$. There is also a different category that one can consider, which we simply denote $\sH_{\le c}$. This is the subcategory of $\sH$ consisting of objects such that all of the simple subquotients of all of their cohomology groups are of the form $L_w$ for $w \le c$. In general, $\sH_{\le c}$ is strictly larger than $\sH_{\le c}^\access$, but every object of $\sH_{\le c}$ which is cohomologically bounded below lies in $\sH_{\le c}^\access$ (see \S\ref{sss:access-vs-all-sl2} for more discussion of this point). 

\medskip 

We should note that the two-sided order is traditionally defined purely combinatorially (see \cite{kazhdan1979representations}). However, Soergel's conjecture implies that the combinatorial definition agrees with the definition in terms of two-sided ideals in $\sH$ (for finite Weyl groups, Soergel's conjecture is a consequence of the decomposition theorem of \cite{BBD}; more generally see \cite{EliasWilliamson2014Hodge}). 

\subsubsection{} The bounded-derived versions of the categories $\sH_{\le c}$ have been studied extensively (see \cite{bezrukavnikov2012character}, \cite{EliasWilliamson2021}). However, an important feature of $\sH_{\le c}$ only appears at the level of unbounded derived categories: the inclusion $\sH_{\le c} \to \sH$ admits a left adjoint (for this to be true it is necessary to consider $\sH_{\le c}$ rather than $\sH_{\le c}^\access$). This was first proved in the de Rham setting in \cite{dhillon2025singularsupportgcategories} using the theory of degenerate Whittaker coefficients.\footnote{Strictly speaking, the results of \textit{loc. cit.} concern the categories $\sH_Z$ obtained by imposing a singular support condition on the objects of $\sH$, but in the de Rham setting one can pass between two-sided cells and singular support conditions using \cite[Theorem 1.1]{BARBASCH1983350}. (See also \cite{tanisaki1988characteristic} for a nice overview.)} In \Cref{prop:semiorthogonal-compat-H-O}, we give a new proof of this result just using basic facts about category $\sO$. It follows that at the level of unbounded derived categories, we obtain a system of semiorthogonal decompositions
\[\sH_{\not\le c} \rightleftarrows \sH \rightleftarrows \sH_{\le c},\]
and by \Cref{constr:semiorthogonal-trace} this produces the desired direct sum decomposition of $\Rep(G(\F_q))^\unip$ into families. 

\medskip 

These semiorthogonal decompositions of $\sH$ are of fundamental importance for us, so let us briefly say what happens for $SL_2$ (see \S\ref{ss:sl2-calculation} for more discussion). The semiorthogonal decomposition of $\sH$ arises from a similar semiorthogonal decomposition of (the unbounded derived category of) the principal block of category $\sO$. If we take $c$ to be the two-sided cell $\set{s}$, where $s$ is the nontrivial element of the Weyl group, the subcategory $(\sO_0)_{\not\le c}$ is just the full subcategory generated under colimits and shifts by the big projective, and $(\sO_0)_{\le c}$ is the derived category of lisse sheaves on $\bP^1$. If $L_1$ is the delta sheaf on the closed Bruhat stratum, then this semiorthogonal decomoposition produces an exact triangle 
\[\tau_{\not\le c}L_1 \to L_1 \to \tau_{\le c}L_1\]
and one can show that 
\[H_i(\tau_{\le c}L_1) = \begin{cases}
L_s & i \ge 1 \\
0 & i < 1
\end{cases}\]  
and
\[H_i(\tau_{\not\le c}L_1) = \begin{cases}
L_s & i > 0 \\
\nabla_s & i = 0 \\
0 & i < 0
\end{cases}\]  
where $\nabla_s$ refers to the costandard object with socle $L_s$. 

\subsubsection{} Now that we have a decomposition 
\[\Rep(G(\F_q))^\unip \simeq \bigoplus_c \Rep(G(\F_q))^\unip_c\]
of the category of unipotent representations into families indexed by two-sided cells, we still need to describe the categories $\Rep(G(\F_q))^\unip_c$. Note that if we pick a representative of each conjugacy class of $A_c$, then we obtain an isomorphism of stacks
\[A_c/A_c \simeq \bigsqcup_{[x] \in \Conj(A_c)} */Z_{A_c}(x),\]
so we can rephrase \S\ref{sss:lusztig-unipotent-reps} as saying that there is an equivalence
\[\Rep(G(\F_q))^\unip_c \simeq \QCoh(A_c/A_c).\]
On the other hand, by a uniform argument we will be able to show that 
\begin{equation}\label{eq:rep-unip-c-group}
\Rep(G(\F_q))^\unip_c \simeq \QCoh_{\sG}(A_c/A_c)    
\end{equation}
for some gerbe $\sG$.\footnote{As we have mentioned already, one can show by case-by-case calculations that $\sG$ is non-canonically trivializable.} We should caution that the group we are calling $A_c$ is defined quite inexplicitly (see \S\ref{sss:intro-losev-ostrik}, \S\ref{sss:characterizing-Ac}). 

\subsubsection{} If $\Gamma$ is a finite group and $\omega \in H^3(\Gamma; \Gm)$ is a cohomology class, then following \cite[\S2.3.3]{quinn1998group} (see also \cite[\S3.1]{Losev_Ostrik_2014}) one can define a monoidal category 
\[\Vect_{\Gamma, \omega}\]
which agrees with $\QCoh(\Gamma)$ as a plain category, but whose associativity constraint is twisted by by $\omega$. By \cite[\S 8]{freed1994higher}, \cite{willerton2008twisted},  
\[HH(\Vect_{\Gamma, \omega}; \id) \simeq \QCoh_{\sG}(\Gamma/\Gamma)\]
for a certain gerbe $\sG$ constructed from $\omega$. In fact, the gerbe on the right hand side of (\ref{eq:rep-unip-c-group}) arises in this way (for a certain cohomology class $\omega$ which we will return to in \S\ref{sss:intro-losev-ostrik}). It follows that we may rewrite (\ref{eq:rep-unip-c-group}) as an equivalence 
\begin{equation}\label{eq:hh-rep-unip-c-group}
HH(\sH_c; \Frob) \simeq HH(\Vect_{A_c, \omega}; \id).    
\end{equation}

\subsubsection{}\label{sss:hh-morita-invt} To prove an equivalence like (\ref{eq:hh-rep-unip-c-group}), it is useful to know that $HH(-)$ is a Morita invariant: if there is an equivalence of categories 
\begin{equation}\label{eq:intro-hh-morita}
\sA \tmod \simeq \sB \tmod,    
\end{equation}
then $HH(\sA) \simeq HH(\sB)$. More generally, if $\sA$ and $\sB$ are equipped with endomorphisms $\phi$ and $\psi$, and the equivalence (\ref{eq:intro-hh-morita}) intertwines the induced endofunctors of module categories,
\[
\begin{tikzcd}
\sA \tmod \arrow[d, "\ind_\phi"'] \arrow[r, "\sim"]& \sB \tmod \arrow[d, "\ind_\psi"]\\
\sA \tmod \arrow[r, "\sim"] & \sB \tmod 
\end{tikzcd}
\]
then $HH(\sA; \phi) \simeq HH(\sB; \psi)$. 

\medskip 

These properties of $HH$ are easy to see from the point of view of categorical traces (see \cite{benzvi2013secondarytraces,kondyrev2022equivariant, gaitsgory2022toymodeldrinfeldlafforgueshtuka}). Namely, for any symmetric monoidal category $\sC$ and any dualizable object $x \in \sC$, there is a map 
\[\tr: \End(x) \to \End(1_\sC)\]
which sends an endomorphism $\phi$ to $\tr(\phi; x)$, defined to be the composite 
\[1_\sC \totext{\epsilon} x \otimes x^\vee \totext{\phi \otimes \id} x \otimes x^\vee \totext{\eta} 1_\sC.\]
For our purposes, one should take $\sC = 2\dgCat$ (see \cite{gaitsgory2026applicationshighercategoricaltrace}) and $x = \sA \tmod$. Every $(\sA, \sA)$-bimodule $\sM$ defines an endomorphism of $\sA \tmod$, and \cite[\S3.7.1]{gaitsgory2022toymodeldrinfeldlafforgueshtuka} shows that 
\[HH(\sA; \sM) \simeq \tr(\sM; \sA\tmod).\]

\subsubsection{} Our proof of (\ref{eq:hh-rep-unip-c-group}) passes through an intermediate object, Lusztig's asymptotic category $\sJ_c$ (see \cite{LUSZTIG199785}, \cite[\S5]{EliasWilliamson2021}). This is a semisimple monoidal category, whose simple objects $J_w$ are labelled by elements $w \in c$. To describe the monoidal structure, we first have to recall the definition of Lusztig's $a$-function. By definition, for a two-sided cell $c$, the number $a(c)$ is the smallest integer such that 
\[H^{i}(L_w * L_v) \in \sH_{< c}\] 
for all $i > a(c)$ and $w, v \in c$. The monoidal structure on $\sJ_c$ is defined so that 
\[J_w * J_v \simeq \bigoplus_{u \in c} J_u^{\oplus n_u}\]
where $n_u$ is the multiplicity of $L_u$ in $H^{a(c)}(L_w * L_v)$. Roughly speaking, $\sJ_c$ remembers the ``leading term'' of the monoidal structure of $\sH_c$. 

\subsubsection{}\label{sss:intro-losev-ostrik} We should now explain how $\sJ_c$ is related to both $\sH_c$ and the category $\Vect_{A_c, \omega}$ (which we have not yet fully defined!). Let's begin with the relationship between $\sJ_c$ and $\Vect_{A_c, \omega}$, which is supplied by a theorem of I. Losev and V. Ostrik \cite[Theorem 5.1]{Losev_Ostrik_2014}. Using the theory of finite $\sW$-algebras, Losev and Ostrik construct a module category $\sM_c$ for $\sJ_c$, with the property that 
\[\End_{\sJ_c}(\sM_c)\]
is a \textit{group category}, in the sense of \cite{quinn1998group} (this means that the category is semisimple and every simple object is tensor invertible). This implies (by \S2.3.2 of \textit{loc. cit.}) that 
\[\End_{\sJ_c}(\sM_c) \simeq \Vect_{A_c, \omega}\]
for some group $A_c$ and some $\omega \in H^3(A_c; \Gm)$. \textit{We will take this as our definition of $A_c$}. Note that the group we call $A_c$ is denoted $\ol{\A}$ in \cite{Losev_Ostrik_2014}; they reserve the notation $A_c$ for a group which is defined in terms of Springer theory (see \S\ref{sss:characterizing-Ac} for more discussion). In fact, Losev and Ostrik prove that $\sM_c$ defines a Morita equivalence 
\[\sJ_c \tmod \simeq \Vect_{A_c, \omega} \tmod.\]
Because $HH(-)$ is a Morita invariant, this implies that we may reformulate (\ref{eq:hh-rep-unip-c-group}) as the following equivalence, which is the main result of this paper 
\begin{theorem}[\Cref{thm:hh-hc-jc}]\label{thm:intro-hh-Hc-Jc}
\[HH(\sH_c; \Frob) \simeq HH(\sJ_c; \id).\]  
\end{theorem}

\subsubsection{}\label{sss:intro-a1-on-hecke} It is in some sense obvious from the definition that the asymptotic category $\sJ_c$ is related to the Hecke category $\sH_c$. However, one of the contributions of this paper is to articulate a new relationship between the two, which in fact supplies a new construction of $\sJ_c$ as a monoidal category. Namely, in \cite{gaitsgory2026excursionalgebra}, D. Gaitsgory, K. Lin, and the author showed that the theory of weights gives rise to an action of the monoid $\A^1$ on a category $\Sh(X)^\relev$ for any scheme $X$ over $\F_q$. Here the monoid structure on $\A^1$ is given by multiplication and the category $\Sh(X)^\relev$ is a certain full subcategory of the category $\Sh(X)$ of ind-constructible $\ol{\Q}_\ell$-sheaves on $X$. 

\medskip 

It is not hard to extend this construction to Artin stacks, and when $X = B \bs G /B$ every ind-constructible sheaf lies in $\Sh(X)^\relev$. We therefore obtain an action of $\A^1$ on $\sH$, and in \S\ref{ss:asymptotic-cat} we show that this induces an action of $\A^1$ on $\sH_c$ such that $\sJ_c$ is simply the category of invariants, 
\[\sJ_c \simeq \sH_c^{\A^1}.\]
We think of this as a precise articulation of the slogan that says that $\sJ_c$ is the limit of $\sH_c$ as $v$ goes to zero.

\subsubsection{}\label{sss:toy-model-localization} With the results of the previous paragraphs in mind, it remains to prove an equivalence 
\begin{equation}\label{eq:hh-localization}
HH(\sH_c; \Frob) \simeq HH(\sH_c^{\A^1}; \id).
\end{equation}
This occupies the bulk of our paper. Let's describe a toy model for our argument: if $A$ is a non-negatively graded algebra and $\act_\lambda$ is the automorphism of $A$ that rescales the $i$th weight space by $\lambda^i$, then whenever $\lambda$ isn't a root of unity, 
\begin{equation}\label{eq:localization-toy-model}
HH(A; \act_\lambda) \simeq HH(A_0; \id),    
\end{equation}
where $A_0$ is the zeroth graded piece of $A$. The argument is simple: note that $A$ is acted on by the monoid $\A^1$, and $A_0 = A^{\A^1}$ is the ring of fixed points for this action. Then because $\A^1$ is commutative, the $\A^1$-action on $A$ induces an action on $HH(A; \act_\lambda)$, and the fixed points of this action are 
\[HH(A_0; \act_\lambda) = HH(A_0; \id).\] 
Now note the action of $\lambda$ on $HH(A; \act_\lambda)$ is canonically trivial. As a consequence, whenever $\lambda$ is not a root of unity (and therefore the powers of $\lambda$ are Zariski dense in $\A^1$), the action of $\A^1$ on $HH(A; \act_\lambda)$ is trivial, so $HH(A; \act_\lambda)$ is equal to its $\A^1$-fixed points. 

\medskip 

To relate our desired equivalence (\ref{eq:hh-localization}) to this toy model, we need to explain why the Frobenius automorphism of $\sH_c$ has the form $\act_\lambda$ for some $\lambda$. In \S\ref{ss:motives} we show that this follows from W. Soergel and M. Wendt's motivic construction of the Hecke category \cite{soergel2018perverse}. As a result, roughly speaking, our equivalence (\ref{eq:hh-localization}) is just a categorified version of the equivariant localization isomorphism (\ref{eq:localization-toy-model}). However, several problems arise when we try to make this precise. 

\subsubsection{} The most serious problem with viewing the equivalence (\ref{eq:hh-localization}) as an instance of equivariant localization is the fact that the $\A^1$-action on $\sH_c$ of \S\ref{sss:intro-a1-on-hecke} is only laxly compatible with the monoidal structure on $\sH_c$. As a result, it is unclear why there should be an actual $\A^1$-action on $HH(\sH_c; \Frob)$. However, we are able to get around this problem by introducing a new monoidal category 
\[\sB_c = \gr_0(1_c)\bimod(\sH_c)\]
such that 
\begin{enumerate}
    \item $\sB_c^{\A^1} \simeq \sJ_c$. (See \Cref{prop:Bc-A1-invts-Jc})
    \item The natural $\A^1$ action on $\sB_c$ is strictly compatible with the monoidal structure. (See \Cref{prop:A1-action-strict})
    \item $HH(\sH_c; \Frob) \simeq HH(\sB_c; \Frob)$. (See \Cref{prop:hh-hc-equals-hh-bc}) 
\end{enumerate}
Here, $\gr_0(1_c)$ is a certain algebra object in $\sH_c$ defined in terms of the $\A^1$-action on $\sH_c$. 

\subsubsection{} Unfortunately for the purposes of proving point (3) of the previous paragraph, the category $\sB_c$ is not quite Morita equivalent to $\sH_c$. Nevertheless, if we let 
\[\sH_c^\access = \sH_c \otimes_{\sH_{\le c}} \sH_{\le c}^{\access}\] 
then we show in \S\ref{ss:accessible} that the natural inclusion $\sH_c^\access \to \sH_c$ admits a continuous right adjoint, and $\sH_c^\access$ admits a monoidal structure such that this right adjoint is strictly monoidal. In \Cref{prop:Bc-morita-equiv} we show that the action of $\sH_c$ on $\gr_0(1_c)\tmod(\sH_c)$ factors through $\sH_c^\access$ and $\gr_0(1_c)\tmod(\sH_c)$ defines a Morita equivalence between $\sB_c$ and $\sH_c^\access$. Finally, in \Cref{prop:tr-inaccess}, using the equivalence 
\[HH(\sH; \Frob) \simeq \Rep(G(\F_q))^\unip\] 
we show that the map 
\[\sH_c \to \sH_c^\access\]
induces an isomorphism after passing to $HH(-; \Frob)$. 

\subsubsection{} The final difficulty that arises when trying to implement the argument of \S\ref{sss:toy-model-localization} one category level up is that dg categories are much less ``rigid'' than vector spaces: in general, if the action of $\A^1$ on a dg category $\sC$ is trivializable on a Zariski dense subset, there's no reason to expect that the whole action of $\A^1$ on $\sC$ is trivializable. However, if $\sC$ is semisimple and the coaction functor 
\[\coact: \sC \to \sC \otimes \QCoh(\A^1)\]
is strongly continuous (meaning that it admits a continuous right adjoint), it is actually true that the action of $\A^1$ on $\sC$ is trivializable. 

\medskip 

Annoyingly, it is not \textit{a priori} obvious that the coaction of $\A^1$ on $HH(\sB_c; \Frob)$ is strongly continuous. However, in \Cref{prop:str-cts-semisimple-conservative} we show that to prove the triviality of the $\A^1$-action it suffices to show that $HH(\sB_c; \Frob)$ is an $\A^1$-equivariant colocalization of a category whose coaction functor is strongly continuous. We then supply the needed colocalization by analyzing a certain renormalization of $\sB_c$ (see \S\ref{ss:Bc-ren}). 

\subsubsection{}\label{sss:intro-lusztig-hom} To summarize our proof of \Cref{thm:intro-hh-Hc-Jc}, we have constructed maps 
\begin{equation}\label{eq:categorification-lusztig-hom}
\sH_c\tmod \totext{\gr_0(1_c)\tmod(\sH_c) \otimes_{\sH_c} (-)} \sB_c \tmod \totext{\ind_{\gr_\sbullet}} \sJ_c \tmod
\end{equation}
and shown that they induce isomorphisms 
\[HH(\sH_c; \Frob) \simeq HH(\sB_c; \Frob) \simeq HH(\sJ_c; \id)\]
after applying categorical trace. We would like to suggest that the composite map (\ref{eq:categorification-lusztig-hom}) should be viewed as a 2-categorification of Lusztig's map from the Hecke algebra to the asymptotic algebra (see \cite[\S2.4]{LUSZTIG1987536}, \cite[Theorem 9.2]{Curtis1988}). We remind that the asymptotic algebra $J_c$ is $K^0(\sJ_c)$. It has a distinguished basis $\set{j_w}$ given by the K-theory classes of the simple objects $J_w$. There is a naive map 
\[\gr_c: H \to J_c[v^{\pm 1}]\]
of $\Z[v^{\pm 1}]$-modules which sends 
\[b_w \mapsto \begin{cases}
j_w & w \in c \\
0 & w \notin c
\end{cases}.\]
This naive map is basically never a ring homomorphism. However, Lusztig observed that if we write 
\[e_c = \sum_{d \in \sD_c} b_d\] 
where $\sD_c$ refers to the set of Duflo involutions in the two-sided cell $c$, then the composite
\begin{equation}\label{eq:lusztig-hom}
H \totext{e_c * (-)} H \totext{\gr_c} J_c[v^{\pm 1}]    
\end{equation}
is actually a map of rings. In \Cref{prop:jc-unit}, we show that
\[\gr_0(1_c) \simeq \bigoplus_{d \in \sD_c} L_d[-a]\]
so the maps (\ref{eq:categorification-lusztig-hom}) and (\ref{eq:lusztig-hom}) seem quite clearly related. However, we do not know how to obtain the latter map as a decategorification of the former. 

\subsubsection{Relationship to \cite{lusztig2015unipotent}}\label{sss:lusztig-paper} Let's recall Lusztig's definition of the partition of irreducible unipotent representations into families indexed by two-sided cells from \cite[\S 8]{lusztig1979unipotent}. If we write $DL_w$ for the character of the unipotent Deligne-Lusztig representation corresponding to $w$, then for any irreducible representation $V$ of $W$ with character $\chi$, we can define a class function for $G(\F_q)$ by the formula 
\[DL_\chi = \frac{1}{|W|} \sum_{w \in W} \chi(w) DL_w.\]

\medskip 

The irreducible representations of $W$ are partitioned into families indexed by two-sided cells as follows: the two-sided cells determine a filtration of the Hecke algebra $H$ by two-sided ideals, and evaluating at $v = 1$ gives a filtration of the group algebra of $W$ by two-sided ideals. A representation of $W$ then belongs to the family corresponding to a two-sided cell $c$ if it arises in the isotypic decomposition of the $c$-graded component of $k[W]$.\footnote{Strictly speaking,  families of $W$-representations are defined in \cite[\S4.2]{lusztig1984characters} somewhat differently, using parabolic induction and generic degrees of $W$-representations. However, a theorem of Barbasch and Vogan \cite[Theorem 2.29]{BARBASCH1983350} (see also \cite[Theorem 5.25]{lusztig1984characters}) shows that these definitions agree. } 
Lusztig's disjointness theorem \cite[Theorem 6.17]{lusztig1984characters} states that if $\rho$ is the character of an irreducible unipotent representation of $G(\F_q)$, and $\chi$, $\chi'$ are irreducible characters of $W$ such that 
\[\bracket{\rho, DL_\chi} \neq 0 \quad\text{ and }\quad \bracket{\rho, DL_{\chi'}} \neq 0\]
then $\chi$ and $\chi'$ belong to the same family. Since \cite{deligne1976representations} shows that every unipotent $\rho$ has a nonzero inner product with some $DL_w$, Lusztig partitions the irreducible unipotent representations $\rho$ into families indexed by two-sided cells by stipulating that $\rho$ lies in the family corresponding to $c$ if it has a nonzero pairing with $DL_\chi$ for some $\chi$ in $\Rep(W)_c$.

\medskip 

In \cite{lusztig2015unipotent}, using this definition of $\Rep(G(\F_q))^\unip_c$ Lusztig shows that 
\begin{equation}\label{eq:lusztig-center}
\Rep(G(\F_q))^\unip_c \simeq Z(\sJ_c)    
\end{equation}
where $Z(-)$ refers to the categorical center. To do this, he uses the theory of \cite{lusztig1984characters} (most notably \cite[Theorem 3.8]{lusztig1984characters} and \cite[Theorem 6.17]{lusztig1984characters}) to show that $\Rep(G(\F_q))^\unip_c$ is monadic over $\sJ_c$. By the results of \cite{etingof2005fusion}, $Z(\sJ_c)$ is also monadic over $\sJ_c$, and Lusztig proves (\ref{eq:lusztig-center}) by showing that the two monads are isomorphic.

\medskip 

Because $\sJ_c$ is a pivotal, rigid monoidal category (see \cite[Theorem 5.2]{EliasWilliamson2021}), there is a canonical equivalence 
\begin{equation}\label{eq:Jc-pivotal}
Z(\sJ_c) \simeq HH(\sJ_c)    
\end{equation}
(see \cite[\S3]{benzvi2015charactertheorycomplexgroup}). We expect, but cannot currently prove that:
\begin{enumerate}
    \item Our definition of $\Rep(G(\F_q))^\unip_c$ agrees with Lusztig's (perhaps up to replacing $c$ with $w_0c$).
    \item With respect to the equivalence (\ref{eq:Jc-pivotal}), our equivalence \Cref{thm:intro-hh-Hc-Jc} identifies with Lusztig's equivalence (\ref{eq:lusztig-center}). 
\end{enumerate}

\subsection{Future directions}

\subsubsection{Lusztig's character formula} It follows from a combination of \Cref{prop:semi-rigid-duality-left-right-mod}, \Cref{prop:Bc-semi-rigid}, and \Cref{prop:Jc-rigid} that in the diagram 
\[
\begin{tikzcd}
& \sB_c \tmod \arrow[dl, "\gr_0(1_c)\tmod^r(\sH_c) \otimes_{\sB_c} (-)"'] \arrow[dr, "\ind_{\gr_\sbullet}"]& \\
\sH_c \tmod & & \sJ_c \tmod
\end{tikzcd}
\]
all of the objects are 2-dualizable and all of the arrows are 2-adjointable (in the sense of \cite{gaitsgory2026applicationshighercategoricaltrace}). Moreover, the arrows in the diagram intertwine the pair of commuting endomorphisms $(\id, \Frob)$ of $\sB_c \tmod$ with the pairs $(\id, \Frob)$ of $\sH_c \tmod$ and $(\id, \id)$ of $\sJ_c \tmod$. It follows that we obtain a commuting diagram 
\[
\begin{tikzcd}
\tr(\id; \tr(\Frob; \sH_c\tmod)) \arrow[d, "S"] & \tr(\id; \tr(\Frob; \sB_c\tmod)) \arrow[l] \arrow[r] \arrow[d, "S"] & \tr(\id; \tr(\id; \sJ_c\tmod)) \arrow[d, "S"] \\
\tr(\Frob; \tr(\id; \sH_c\tmod)) & \tr(\Frob; \tr(\id; \sB_c\tmod)) \arrow[l] \arrow[r] & \tr(\id; \tr(\id; \sJ_c\tmod)) 
\end{tikzcd}
\] 
where the vertical arrows are the isomorphisms constructed in \cite{benzvi2013secondarytraces} (see also \cite{gaitsgory2026applicationshighercategoricaltrace}), and the horizontal arrows in the upper row are isomorphisms by \Cref{prop:tr-inaccess} and \Cref{prop:hh-bc-equals-hh-jc}. By \S\ref{sss:intro-losev-ostrik}, 
\[\tr(\id; \tr(\id; \sJ_c\tmod)) \simeq \tr(\id; \tr(\id; \Vect_{A_c, \omega}\tmod)) \simeq \Gamma(Z^{(2)}_{A_c}; \sL)\]
where $\sL$ is some line bundle depending on $\omega$ and $Z^{(2)}_{A_c}$ refers to the ``commuting stack'' of pairs of commuting elements of $A_c$ mod conjugacy. The line bundle $\sL$ is known by case-by-case computations to be trivializable, see \S\ref{sss:gerbe-trivalize}. With respect to this identification, $S$ acts on $\Gamma(Z^{(2)}_{A_c}, \sO)$ via the automorphism of $Z^{(2)}_{A_c}$ that swaps the elements in the commuting pair (up to a scalar factor if $\omega \in H^3(A_c; \Gm)$ is nontrivial), see \cite{freed1993chern}. It should be possible to rederive Lusztig's character formula from this point of view, as predicted by Ben-Zvi and Nadler (see \cite[Remark 1.5]{benzvi2013secondarytraces}). There are two main difficulties that arise: first, before passing to trace of Frobenius, the arrows 
\[\begin{tikzcd}
 & \tr(\id; \sB_c\tmod) \arrow[dl] \arrow[dr] & \\
\tr(\id; \sH_c\tmod) & & \tr(\id; \sJ_c\tmod)    
\end{tikzcd}\] 
are not isomorphisms. Second, in the diagram 
\[\sH_{\not\le c} \tmod \rightleftarrows \sH \tmod \rightleftarrows \sH_{\le c} \tmod\]
only the rightward arrows are 2-adjointable, so the vector space
\[\tr(\id; \tr(\Frob; \sH_c\tmod)) \simeq \tr(\Frob; \tr(\id; \sH_c\tmod))\]
is \textit{a priori} only a subquotient of the space of class functions on $G(\F_q)$. We hope to address both of these points in a future publication.

\subsubsection{Characterizing $A_c$}\label{sss:characterizing-Ac} An unsatisfactory feature of the analysis in this paper is the fact that for us the group $A_c$ is defined very inexplicitly. In fact, Losev and Ostrik show \cite[Theorem 7.4]{Losev_Ostrik_2014} that the group $A_c$ can be completely characterized as follows: in characteristic zero, there is a bijection between two-sided cells and certain nilpotent orbits called special nilpotent orbits. Let $\bO$ denote the special nilpotent orbit associated to $c$, and let $\Spr(\bO)$ denote the Springer representation attached $\bO$. We can then consider the subrepresentation 
\[\Spr(\bO)_c \subseteq \Spr(\bO)\] 
consisting of the isotypic components in the family $\Irr(W)_c$ corresponding to $c$; the group $A_c$ can then be characterized as the smallest quotient of the component group $A(\bO)$ through which the action on $\Spr(\bO)_c$ factors.\footnote{The fact that $\Rep(G(\F_q))^\unip_c$ can be described as $\QCoh(A_c/A_c)$, where $A_c$ is defined in terms of Springer theory as above, was already observed in \cite[\S13.1]{lusztig1984characters} prior to the introduction of the asymptotic algebra.}

\medskip 

Unfortunately for our goal of finding a uniform account of the theory of unipotent representations, Losev and Ostrik's characterization of $A_c$ relies on many type-by-type calculations. However, recently in the work of G. Dhillon and J. Faergeman \cite[Theorem 1.10.8.1]{dhillon2025singularsupportgcategories} the theory of categorical traces has led to simplifications in certain parts of the Losev and Ostrik's argument. We hope to investigate the theory of the aysmptotic algebra futher from the perspective of $\A^1$ actions.

\subsubsection{Quasisplit and Suzuki-Ree groups}

Although we have restricted for simplicity to the case of split groups, we expect that the results of this paper on the relationship between the Hecke category and the asymptotic category should extend to the cases of quasisplit and Suzuki-Ree groups with only minor adjustments: most importantly, the Steinberg endomorphism of $G$ induces a nontrivial endomorphism $F$ of the asymptotic category. For quasisplit groups, the equivalence 
\[\sJ_c \tmod \simeq \Vect_{A_c} \tmod \]
of \S\ref{sss:intro-losev-ostrik} should acquire a natural $F$-equivariant structure. However, one should \textit{not} expect this for the Suzuki-Ree groups. 

\medskip

For example, consider the case $G = B_2$ with $F$ the endomorphism giving rise to the Suzuki group. If we take 
\[c = \set{s, t, st, ts, sts, tst}\] 
to be the two-sided cell corresponding to the subregular nilpotent orbit, then $F$ acts on $\sJ_c$ by swapping $s$ and $t$. The group $A_c$ is $\Z/2$, and the category $\sM_c$ of \S\ref{sss:intro-losev-ostrik} is 
\[\sM_c = \QCoh(Y_c),\]
where the $A_c$-set $Y_c$ is $\set{L, M, N}$. Here the action of $A_c$ fixes $L$ and swaps $M$ and $N$. Under the equivalence 
\[\sJ_c \simeq \QCoh_{A_c}(Y_c \times Y_c)\]
of \cite{bezrukavnikov2009tensor,Losev_Ostrik_2014}, the object $J_s$ goes to the trivial representation of the stabilizer of $(L, L)$, while the object $J_t$ goes to the constant sheaf on the orbit $\set{(M, M), (N, N)}$. It follows that there is no automorphism of $Y_c$ which is compatible with the Suzuki endomorphism of $\sJ_c$ with respect to the action of $\sJ_c$ on $\sM_c$. 

\medskip 

In fact, \cite[p. 373, 2nd Table]{lusztig1984characters} (see also \cite{BRUNAT2006869}) shows that the family of unipotent representations of the Suzuki group associated to the cell $c$ has only two elements. In particular, it cannot arise as $\QCoh(\Gamma/\Gamma)$ for any finite group $\Gamma$. 

\subsubsection{Other Coxeter groups} Even though the Hecke category $\sH$ and the two-sided cell filtration make sense for any Coxeter group $W$, in this paper we have only investigated the case where $W$ is a finite Weyl group. Let us highlight a few places in our analysis where we expect that the results should generalize to other Coxeter groups, even though our current proof relies on the fact that $W$ is a finite Weyl group. First, our argument that 
\begin{equation}\label{eq:intro-tr-B-tr-J}
\tr(\Frob; \sB_c\tmod) \simeq \tr(\id; \sJ_c\tmod)    
\end{equation}
is semisimple relies on the fact that $\tr(\Frob; \sH_c\tmod)$ is known to be semisimple (via its relationship to $\Rep(G(\F_q))$), however, we suspect that it is possible to prove the identity (\ref{eq:intro-tr-B-tr-J}) for any finite Coxeter group, and deduce semisimplicity as a consequence. A proof along these lines might shed light on the relationship, conjectured in \cite[\S7.2]{lusztig2015unipotent}, between Lusztig's generalized unipotent characters \cite{lusztig1993coxeter} and $HH(\sJ_c)$. 

\medskip 

Additionally, our argument in \S\ref{ss:A1-action-analysis} that the action of $\A^1$ on $\sB_c$ is strictly compatible with the monoidal structure seems far from optimal. One could instead try to prove result this by categorifying Lusztig's proof that the map $\psi_c$ is a ring homomorphism. This strategy would presumably involve analyzing the category $\sH_c^{\Gm \times \Gm}$ of bigraded objects, and would hopefully apply to Coxeter groups which are not finite.

\subsection{Notation and conventions}

\subsubsection{} Our conventions follow those of \cite{beraldo2024coherentsheavessheareddmodules}. 

\medskip

Let's highlight that for a category $\sA$ we use the notation $HH(\sA)$ to refer to the category
\[\sA \underset{\sA \otimes \sA^\rev}{\otimes} \sA.\]
In particular, for us $HH(-)$ does not change the categorical level. On the other hand, we use the notation $\tr(-)$ to denote the operation of categorical trace, which has the effect of decreasing the categorical level by one. This is in agreement with \cite{arinkin2022stacklocalsystemsrestricted,gaitsgory2026applicationshighercategoricaltrace} but disagrees with the notation of \cite{benzvi2015charactertheorycomplexgroup, zhu2025tamecategoricallocallanglands}.

\subsection{Acknowledgements}

I would like to thank A. Eteve and D. Gaitsgory for several helpful discussions about Lusztig's theory of unipotent representations, in the context of our ongoing work on the representations of finite groups of Lie type. In particular, I would like to thank D. Gaitsgory for suggesting to me that Lusztig's asymptotic category should be related to the theory of weights, and I would like to thank A. Eteve for emphasizing to me the importance of the fact that $\sJ_c$ is independent of $\ell$. Additionally, I would like to thank D. Gaitsgory for carefully reading a draft of this paper and providing many helpful comments. Finally, I must note that it would have been appropriate for A. Eteve to have been a coauthor of this paper, but he humbly declined my offer. 

\medskip

I would additionally like to thank J. Baine for many helpful discussions about Soergel bimodules, and for sharing his work on the Jones-Wenzl idempotent \cite{baine2024coefficientsjoneswenzlidempotent}. I would also like to thank K. Lin: much of my understanding of the theory of weights was developed during our ongoing joint work on Arthur's conjectures. 

\medskip

I would like to thank G. Lusztig for helpful comments on the history of the theory of unipotent representations. 

\medskip 

This work was conducted while the author was working at the Max Planck Institute for Mathematics. I thank them for their hospitality. 

\section{Semi-rigid monoidal categories}

\subsubsection{} The goal of this section is to prove some basic facts about semi-rigid monoidal categories which will be useful for us later. 

\medskip 

The outline of this section is as follows: 
\begin{itemize}
    \item In \S\ref{ss:2-sr-so} we prove a couple of facts about semiorthogonal decompositions of semi-rigid categories. 
    \item In \S\ref{ss:2-tr} we recall some basic facts from the formalism of categorical traces, and study the relationship between semi-rigidity and the formation of categorical traces. 
\end{itemize}

\subsection{Semiorthogonal decompositions}\label{ss:2-sr-so}

\begin{proposition}\label{prop:semiorthogonal-semirigid}
Suppose that $\sA$ is a semi-rigid monoidal category. If 
\[\sA_\ell \rightleftarrows \sA \rightleftarrows \sA_r\]
is a semiorthogonal decomposition of $\sA$, and if either $\sA_\ell$ or $\sA_r$ is a two-sided ideal in $\sA$, then the other category is too, and both $\sA_\ell$ and $\sA_r$ admit unique monoidal structures such that $\sA \to \sA_\ell$ and $\sA \to \sA_r$ are monoidal. 
\end{proposition}

\begin{proof}
We will consider the case when $\sA_r$ is a two-sided ideal. The case when $\sA_\ell$ is a two-sided ideal is completely analogous. Since $\sA_r$ is a two-sided ideal, it acquires a unique $(\sA, \sA)$-bimodule structure such that the inclusion $\sA_r \to \sA$ is bilinear. It follows that 
\[\sA_\ell = \cof(\sA_r \to \sA)\]
acquires a unique $(\sA, \sA)$-bimodule structure such that the projection $\sA \to \sA_\ell$ is bilinear. Since $\sA$ is semi-rigid, the a priori oplax $(\sA, \sA)$-bilinear structures on the inclusion $\sA_\ell \to \sA$ and the projection $\sA \to \sA_r$ are actually strict. 

\medskip 

Because $\sA_\ell$ is a two-sided ideal in $\sA$, it follows from \cite[Proposition 2.2.1.9]{lurie2017higher} (arguing as in \cite[Proposition 4.8.2.7]{lurie2017higher}) that $\sA_r$ admits a unique monoidal structure such that $\sA \to \sA_r$ is monoidal. The same is also true for $\sA \to \sA_\ell$. 
\end{proof}

\begin{proposition}\label{prop:semirigid-loc-coloc}
Suppose that $\sA$ is a semi-rigid monoidal category. If $\sB$ is a monoidal localization or colocalization of $\sA$ then $\sB$ is semi-rigid as well.
\end{proposition}

\begin{proof}
First suppose that $\sB$ is a monoidal localization of $\sA$. Because $\sB$ is a retract of $\sA$, it is dualizable. It remains to show that 
\[\mult: \sB \otimes \sB \to \sB\]
admits a continuous, $\sB$-bilinear right adjoint. Since $\mult$ is $\sA$-bilinear, if $\mult^R$ is continuous then it is automatically $\sA$-bilinear, and hence $\sB$-bilinear. It suffices to check that $\mult^R$ is continuous after post-composing with the continuous, conservative functor $i \otimes \id: \sB \otimes \sB \to \sA \otimes \sB$. The left adjoint of the composite functor is 
\[\sA \otimes \sB \totext{\pi} \sB \otimes \sB \totext{\mult} \sB\]
which identifies with the action of $\sA$ on $\sB$. Because $\sA$ is semi-rigid, this functor automatically has a continuous right adjoint. 

Now suppose that $\sB$ is a monoidal colocalization of $\sA$. Like in the previous case, it's enough to show that $\mult: \sB \otimes \sB \to \sB$ admits a continuous right adjoint. However, this functor can be factored as 
\[\sB \otimes \sB \totext{i \otimes \id} \sA \otimes \sB \totext{\act} \sB,\]
and both functors in the composite admit continuous right adjoints.
\end{proof}

\subsection{Traces}\label{ss:2-tr}

\subsubsection{}\label{sss:trace-functoriality} Recall from \S\ref{sss:intro-hh-defn} that for a monoidal category $\sA$ and $(\sA, \sA)$-bimodule $\sM$, the Hochschild homology $HH(\sA, \sM)$ is defined to be 
\[HH(\sA, \sM) = \sA \underset{\sA \otimes \sA^\rev}{\otimes} \sM.\]
Recall also that the formalism of categorical traces \S\ref{sss:hh-morita-invt} allows us to rewrite 
\[HH(\sA, \sM) \simeq \tr(\sM; \sA \tmod).\]
As a result, the functoriality of categorical traces (see for example \cite{gaitsgory2022toymodeldrinfeldlafforgueshtuka}) gives rise to functoriality for $HH$. 

\medskip 

We will be especially interested in the following functoriality of $HH$: let $\sA$ and $\sB$ be monoidal categories, and let $f: \sA \to \sB$ be a monoidal functor between them. Suppose that $\sA$ and $\sB$ have monoidal endomorphisms $\phi$ and $\psi$ that are intertwined by $f$. Then the functoriality of categorical traces gives us a map $HH(\sA; \phi) \to HH(\sB; \psi)$, which is the composite 
\begin{equation*}
\begin{split}
&\sA \underset{\sA \otimes \sA^\rev}{\otimes} \sA_\phi \to \sB \underset{\sA \otimes \sA^\rev}{\otimes} \sA_\phi \simeq \sB \underset{\sB \otimes \sB^\rev}{\otimes} \left(\sB \underset{\sA}{\otimes} \sA_\phi \underset{\sA}{\otimes} \sB \right) \\
&\simeq \sB \underset{\sB \otimes \sB^\rev}{\otimes} \left(\sB_\psi \underset{\sB}{\otimes} \sB \underset{\sA}{\otimes} \sB \right) \to \sB \underset{\sB \otimes \sB^\rev}{\otimes} \left(\sB_\psi \underset{\sB}{\otimes} \sB \right) \simeq \sB \underset{\sB \otimes \sB^\rev}{\otimes} \sB_\psi 
\end{split}
\end{equation*}
where the first non-invertible arrow is obtained by tensoring up from $f: \sA \to \sB$, and the second non-invertible arrow is obtained by tensoring up from 
\[\mult: \sB \otimes_\sA \sB \to \sB.\] 

\subsubsection{}\label{sss:hh-monadic} Because the multiplication functor 
\[\mult: \sA \otimes \sA \to \sA\]
is $(\sA, \sA)$-bilinear, we may tensor this map with $\sM$ over $\sA \otimes \sA^\rev$ to obtain a canonical map
\[\sM \to HH(\sA, \sM)\]
called the universal trace. When $\sA$ is semi-rigid, $\mult$ admits a continuous $(\sA, \sA)$-bilinear right adjoint, so the universal trace necessarily admits a continuous right adjoint. Additionally, because $\mult$ is generating (because it is after all essentially surjective), it follows that the universal trace is as well. In particular, $HH(\sA, \sM)$ is monadic over $\sM$.

\begin{construction}\label{constr:semiorthogonal-trace}
Suppose that $\sA$ is a semi-rigid monoidal category and 
\[\sA_\ell \rightleftarrows \sA \rightleftarrows \sA_r\]
is a semiorthogonal decomposition such that the arrows are all $(\sA, \sA)$-bilinear. Suppose that $\phi$ is a monoidal endomorphism of $\sA$ that preserves the given semiorthogonal decomposition. Then there is a commutative diagram 
\[
\begin{tikzcd}
\tr(\phi; \sA_\ell\tmod) \arrow[r, shift left=1] \arrow[d, shift left=1] & \tr(\phi; \sA\tmod) \arrow[r, shift left=1] \arrow[l, shift left=1] \arrow[d, shift left=1]& \tr(\phi; \sA_r\tmod) \arrow[l, shift left=1] \arrow[d, shift left=1]\\
\sA_\ell \arrow[r, shift left=1] \arrow[u, shift left=1]& \sA \arrow[r, shift left=1] \arrow[l, shift left=1] \arrow[u, shift left=1]& \sA_r \arrow[l, shift left=1]\arrow[u, shift left=1]
\end{tikzcd}
\]
where the horizontal rows are semiorthogonal decomopositions and the vertical arrows realize the upper categories as monadic over the lower categories. 
\end{construction}

\begin{proof}
Starting with the semiorthogonal decomposition 
\[\sA_\ell \rightleftarrows \sA \rightleftarrows \sA_r\]
of $(\sA, \sA)$-bimodules, we can twist the action of $\sA$ on the right by $\phi$ to obtain a semiorthogonal decomposition 
\begin{equation}\label{eq:phi-semi-orthogonal}
(\sA_\ell)_\phi \rightleftarrows \sA_\phi \rightleftarrows (\sA_r)_\phi
\end{equation}
of $(\sA, \sA)$-bimodules. As in \S\ref{sss:hh-monadic} the adjunction 
\begin{equation}\label{eq:mult-monad}
\mult: \sA \otimes \sA \rightleftarrows \sA : \mult^R
\end{equation}
realizes $\sA$ as monadic over $\sA \otimes \sA$. Tensoring (\ref{eq:phi-semi-orthogonal}) and (\ref{eq:mult-monad}) over $\sA \otimes \sA^\rev$ gives a diagram 
\begin{equation}\label{eq:semiorthogonal-diagram}
\begin{tikzcd}
(\sA_\ell)_\phi \underset{\sA \otimes \sA^\rev}{\otimes} \sA \arrow[r, shift left=1] \arrow[d, shift left=1] & \tr(\phi; \sA\tmod) \arrow[r, shift left=1] \arrow[l, shift left=1] \arrow[d, shift left=1]& (\sA_r)_\phi \underset{\sA \otimes \sA^\rev}{\otimes} \sA \arrow[l, shift left=1] \arrow[d, shift left=1]\\
\sA_\ell \arrow[r, shift left=1] \arrow[u, shift left=1]& \sA \arrow[r, shift left=1] \arrow[l, shift left=1] \arrow[u, shift left=1]& \sA_r \arrow[l, shift left=1]\arrow[u, shift left=1]
\end{tikzcd}
\end{equation}
Because the functor $(-) \otimes_{\sA \otimes \sA^\rev} \sA$ commutes with colimits, 
\[\cof\left((\sA_\ell)_\phi \underset{\sA \otimes \sA^\rev}{\otimes} \sA \to \sA_\phi \underset{\sA \otimes \sA^\rev}{\otimes} \sA\right) \simeq \cof((\sA_\ell)_\phi \to \sA_\phi) \underset{\sA \otimes \sA^\rev}{\otimes} \sA \simeq \sA_r \underset{\sA \otimes \sA^\rev}{\otimes} \sA\]
and it follows that the upper row of (\ref{eq:semiorthogonal-diagram}) is indeed a semiorthogonal decomposition. It remains to identify $(\sA_\ell)_\phi \underset{\sA \otimes \sA^\rev}{\otimes} \sA$ with $\tr(\phi; \sA_\ell \tmod)$ (the proof for $\sA_r$ is the same). Indeed: 
\[(\sA_\ell)_\phi \underset{\sA \otimes \sA^\rev}{\otimes} \sA \simeq (\sA_\ell)_\phi \underset{\sA_\ell \otimes \sA_\ell^\rev}{\otimes} \sA_\ell \otimes \sA_\ell \underset{\sA \otimes \sA^\rev}{\otimes} \sA \simeq (\sA_\ell)_\phi \underset{\sA_\ell \otimes \sA_\ell^\rev}{\otimes} \sA_\ell \underset{\sA}{\otimes} \sA_\ell\]
and the result follows because $\sA_\ell \otimes_\sA \sA_\ell \simeq \sA_\ell$ as an $(\sA_\ell, \sA_\ell)$-bimodule.
\end{proof}

\begin{proposition}\label{prop:hh-map}
The map 
\[HH(\sA; \phi) \to HH(\sA_\ell; \phi)\]
of \Cref{constr:semiorthogonal-trace} identifies with the map from \S\ref{sss:trace-functoriality}.
\end{proposition}

\begin{proof}
Unwinding the definition of the functoriality of $HH$ shows that the map 
\[HH(\sA; \phi) \to HH(\sA_\ell; \phi)\]
identifies with the composite 
\[\sA_\phi \underset{\sA \otimes \sA^\rev}{\otimes} \sA \to \sA_\phi \underset{\sA \otimes \sA^\rev}{\otimes} \sA_\ell \totext{\sim} (\sA_\ell)_\phi \underset{\sA_\ell \otimes \sA_\ell^\rev}{\otimes} \sA_\ell\]
where the second arrow is a certain isomorphism. In particular, both the map coming from the functoriality of $HH$ and the map of \Cref{constr:semiorthogonal-trace} admit fully faithful left adjoints. Therefore, to prove that the maps are isomorphic it suffices to show that 
\[\sA_\phi \underset{\sA \otimes \sA^\rev}{\otimes} \sA_\ell \quad \text{and} \quad (\sA_\ell)_\phi \underset{\sA \otimes \sA^\rev}{\otimes} \sA\]
agree as full subcategories of $HH(\sA; \phi)$. This follows because both subcategories are equal to 
\[(\sA_\ell)_\phi \underset{\sA \otimes \sA^\rev}{\otimes} \sA_\ell.\]
\end{proof}

\subsubsection{} The analogoue of \Cref{prop:hh-map} for $\sA_r$ is true as well, and the proof is the same. 

\section{$\A^1$-actions}

\subsubsection{} In \cite[\S1]{gaitsgory2026excursionalgebra}, D. Gaitsgory, K. Lin, and the author proved the following result about actions of the commutative monoid $\A^1$ (with the monoid structure given by multiplication) on dg categories:

\begin{proposition}[{\cite[Theorem 1.1.7]{gaitsgory2026excursionalgebra}}]\label{prop:A1-action-char}
Suppose that $\sC$ is a category acted on by $\Gm$. Then the following data are equivalent:
\begin{enumerate}
    \item An extension of the given $\Gm$ action to an action of $\A^1$. 
    \item A collection of colocalizations $\set{\sC^\Gm_{\le n}}$ of $\sC^\Gm$ such that 
    \[\sC^\Gm = \colim_n \sC^\Gm_{\le n}\]
    and $\sO(m) \cdot \sC^\Gm_{\le n} = \sC^\Gm_{\le n + m}$ (here $\sO(m)$ denote the weight $m$ representation of $\Gm$).
\end{enumerate}
Additionally, if we write $\sC^\Gm_{> n}$ for the right orthogonal of $\sC^\Gm_{\le n}$, then for two categories $\sC$, $\sD$ acted on by $\A^1$, an $\A^1$-equivariant functor $F: \sC \to \sD$ is the same as a $\Rep(\Gm)$-linear functor $F^\Gm: \sC^\Gm \to \sD^\Gm$ such that $\sC^\Gm_{\le 0}$ maps to $\sD^\Gm_{\le 0}$ and $\sC^\Gm_{> 0}$ maps to $\sD^\Gm_{> 0}$. 
\end{proposition}

\subsubsection{} The goal of this section is to prove a few more results about actions of $\A^1$ on dg categories which will be useful to us. Many of these results are based on conversations with K. Lin.

\medskip 

The outline of this section is as follows: 

\begin{itemize}
    \item In \S\ref{ss:3-inv} we give a simple criterion to recognize when a category $\sD$ is the category of $\A^1$-invariants for a category $\sC$ acted on by $\A^1$. 
    \item In \S\ref{ss:3-hh} we study the interaction of $\A^1$-actions with the formation of Hochschild homology. 
    \item In \S\ref{ss:3-lax} we give a characterization of when a functor between categories acted on by $\A^1$ admits a left- or right-lax $\A^1$-equivariant structure. 
    \item In \S\ref{ss:3-monad} we study $\A^1$ actions on categories which are monadic over a category with a preexisting action of $\A^1$.  
\end{itemize}

\subsection{Invariants}\label{ss:3-inv}

\subsubsection{} Let $\sC$ be a category acted on by $\A^1$. The goal of this section is to prove a simple characterization of the category of $\A^1$-invariants $\sC^{\A^1}$. First we note that because 
\[0 \cdot \lambda = \lambda \cdot 0 = 0\]
for all $\lambda \in \A^1$, the map 
\[\act_0: \sC \to \sC\]
naturally intertwines the \textit{trivial} $\A^1$-action on the source with the given $\A^1$-action on the target, and it intertwines the given $\A^1$-action on the source with the trivial $\A^1$-action on the target. In particular, by the universal property of the category of invariants, there is a unique map $r: \sC \to \sC^{\A^1}$ making the diagram 
\[
\begin{tikzcd}
& \sC^{\A^1} \arrow[dr, "i"]& \\
\sC \arrow[rr, "\act_0"] \arrow[dashed, ur, "r"] & & \sC 
\end{tikzcd}
\]
commute, where $i: \sC^{\A^1} \to \sC$ is the canonical map. 

\begin{proposition}\label{prop:invts-idempotent}
The maps $i$ and $r$ realize $\sC^{\A^1}$ as the ``subobject'' of $\sC$ corresponding to the idempotent endomorphism $\act_0$ of $\sC$. 
\end{proposition}

\begin{proof}
Since $i \circ r \simeq \act_0$, it remains to show that $r \circ i \simeq \id_{\sC^{\A^1}}$. By the universal property of $\sC^{\A^1}$, it suffices to show that 
\[iri \simeq i\]
as $\A^1$-equivariant maps from $\sC^{\A^1}$ to $\sC$. This follows from the commutativity of the square 
\[
\begin{tikzcd}
(\sC^{\A^1})^\triv \arrow[r, "\act_0"] \arrow[d, "i"]& \sC^{\A^1} \arrow[d, "i"] \\
\sC^\triv \arrow[r, "\act_0"] & \sC
\end{tikzcd}
\]
and the fact that, because the $\A^1$-action on $\sC^{\A^1}$ is already trivial, the map 
\[\act_0: (\sC^{\A^1})^\triv \to \sC^{\A^1}\]
is isomorphic to the identity. 
\end{proof}

\begin{corollary}\label{prop:A1-invts-char}
Suppose that $\sC$ is a category acted on by $\A^1$. Suppose that we are given a category $\sD$ and morphisms $i: \sD \to \sC$, $r: \sC \to \sD$ such that the following diagrams commute 
\[
\begin{tikzcd}
& \sD \arrow[rd, "i"]& \\
\sC \arrow[rr, "\act_0"'] \arrow[ru, "r"] & & \sC
\end{tikzcd}
\quad\quad\quad
\begin{tikzcd}
\sD \arrow[dr, "i"'] \arrow[rr, "\id_\sD"] & & \sD \\
& \sC \arrow[ur, "r"']&
\end{tikzcd}
\]
Then $\sD$ is equivalent to $\sC^{\A^1}$.  
\end{corollary}

\begin{proof}
Follows immediately from \Cref{prop:invts-idempotent} the uniqueness of ``subobjects'' corresponding to idempotents.  
\end{proof}

\subsubsection{} As a final remark, let's note that $r: \sC \to \sC^{\A^1}$ has a natural $\A^1$-equivariance data, where the source has the given $\A^1$ action. This is because $r$ is isomorphic to $r \circ \act_0$, and we can write the latter as a composite of $\A^1$-equivariant functors as follows:
\[\sC \totext{\act_0} \sC^\triv \totext{r} \sC^{\A^1}.\]

\subsection{Hochschild homology}\label{ss:3-hh}

\subsubsection{}\label{sss:hh-A1-equiv} Let's start by noting that one can perform the constructions of \S\ref{sss:trace-functoriality} internally to a general symmetric monoidal category $\sC$. (The material in \S\ref{sss:trace-functoriality} concerns the case when $\sC = \dgCat$). Namely, an for associative algebra object $A \in \sC$, one can form $HH(A) \in \sC$, which is the object 
\[A \underset{A \otimes A^\rev}{\otimes} A.\]
More generally, if $A$ admits an algebra endomorphism $\phi$, one can form the twisted Hochschild homology 
\[HH(A; \phi) = A \underset{A \otimes A^\rev}{\otimes} A_\phi.\]
Aditionally, if $A$ and $B$ are algebra objects with endomorphisms and $f: A \to B$ is a map of algebras which intertwines these endomorphisms, then we obtain a map 
\[HH(A; \phi_A) \totext{HH(f)} HH(B; \phi_B).\] 

\medskip 

Applying this paradigm to the situation where $\sC = \A^1\textrm{-Cat}$, the symmetric monoidal category of categories acted on by $\A^1$, we see that if $\sA$ is an $\A^1$-equivariantly monoidal category and $\phi$ is an $\A^1$-equivariant monoidal automorphism of $\sA$, then the category $HH(\sA; \phi)$ acquires a natural $\A^1$-action. Concretely, for $\lambda \in \A^1$, the endomorphism $\act_\lambda$ of $HH(\sA; \phi)$ identifies with 
\[HH(\act_\lambda): HH(\sA; \phi) \to HH(\sA; \phi).\]

\begin{proposition}\label{prop:invts-hh}
In the above setup, the action of $\A^1$ on $HH(\sA; \phi)$ has the property that 
\[HH(\sA; \phi)^{\A^1} \simeq HH(\sA^{\A^1}; \phi).\]
\end{proposition}

\begin{proof}
Follows from \Cref{prop:A1-invts-char} by applying $HH(-)$ to the triangles 
\[
\begin{tikzcd}
& \sA^{\A^1} \arrow[rd, "i"]& \\
\sA \arrow[rr, "\act_0"'] \arrow[ru, "r"] & & \sA
\end{tikzcd}
\quad\quad\quad
\begin{tikzcd}
\sA^{\A^1} \arrow[dr, "i"'] \arrow[rr, "\id_\sD"] & & \sA^{\A^1} \\
& \sA \arrow[ur, "r"']&
\end{tikzcd}
\]
\end{proof}

\begin{proposition}\label{prop:hh-strong-cts-coact}
In the setup above, if $\sA$ is semi-rigid and the $\A^1$ action on $\sA$ has the property that 
\[\coact: \sA \to \sA \otimes \QCoh(\A^1)\]
admits a continuous right adjoint, then the same is true for the action of $\A^1$ on $HH(\sA; \phi)$. 
\end{proposition}

\begin{proof}
Recall from \S\ref{sss:hh-monadic} that $HH(\sA; \phi)$ is monadic over $\sA$, so we may check that $\coact^R$ is continuous after applying $\oblv: HH(\sA; \phi) \to \sA$. Now the result follows from the commutativity of 
\[
\begin{tikzcd}
HH(\sA; \phi) \arrow[r] & HH(\sA; \phi) \otimes \QCoh(\A^1) \\
\sA \arrow[r] \arrow[u] & \sA \otimes \QCoh(\A^1) \arrow[u]    
\end{tikzcd}
\]
by passing to right adjoints. 
\end{proof}

\subsection{Lax equivariance}\label{ss:3-lax}

\subsubsection{} Let $\sC$ and $\sD$ be categories acted on by $\A^1$. We will say that a map $F: \sC \to \sD$ has a left-lax (resp. right-lax) $\A^1$-equivaraint structure if $F$ has the structure of a left-lax (resp. right-lax) linear map of $\QCoh(\A^1)$-modules. 

\begin{proposition}\label{prop:left-lax-preserve-le-0}
Suppose that $F: \sC \to \sD$ is left-lax $\A^1$-equivariant. Then the image of $\sC^\Gm_{\le n}$ under $F^\Gm$ is contained in $\sD_{\le n}^\Gm$. 
\end{proposition}

\begin{proof}
Under the equivalence of \Cref{prop:A1-action-char}, the category $\sC^\Gm_{\le n}$ identifies with 
\[\sO(n)\textrm{-comod}(\sC^\Gm)\] 
where $\sO(n) \in \QCoh(\A^1/\Gm)$ corresponds under the equivalence $\QCoh(\A^1/\Gm)\simeq \Vect^\Fil$ with the filtered vector space $V_\sbullet$ with 
\[V_i = \begin{cases}
k & i \ge -n \\
0 & \text{otherwise}
\end{cases}
\]
Since the functor $\sC^\Gm \to \sD^\Gm$ is assumed to be left-lax $\QCoh(\A^1/\Gm)$-linear, it induces a functor 
\[\sO(n)\textrm{-comod}(\sC^\Gm) \to \sO(n)\textrm{-comod}(\sD^\Gm)\]
which is what we wanted to show. 
\end{proof}

\begin{construction}\label{con:preserve-le-0-left-lax}
Suppose $\sC$ and $\sD$ are categories acted on by $\A^1$ and suppose that $F: \sC^\Gm \to \sD^\Gm$ has the property that $F^\Gm: \sC^\Gm \to \sD^\Gm$ sends $\sC^\Gm_{\le n}$ to $\sD^\Gm_{\le n}$. Then $F$ admits a structure of left-lax $\A^1$-equivariance. 
\end{construction}

\begin{proof}
Consider the comma category $\sD^\Gm_{F^\Gm/-}$, viewed as a $\Rep(\Gm)$-module category. The objects of $\sD^\Gm_{F^\Gm/-}$ consist of triples $(x, y, \alpha)$, where $x \in \sC^\Gm$, $y \in \sD^\Gm$, and $\alpha$ is a map $F^\Gm(x) \to y$. This category comes equipped with continuous and jointly conservative functors 
\begin{equation}\label{eq:projectors}
\begin{split}
\sD^\Gm_{F^\Gm/-} \to \sC^\Gm \\
\sD^\Gm_{F^\Gm/-} \to \sD^\Gm
\end{split}
\end{equation}
sending a triple $(x, y, \alpha)$ to $x$ and $y$, respectively. 

\medskip 

If we set $\sD^\Gm_{F^\Gm/-, \le 0}$ to be the full subcategory of $\sD^\Gm_{F^\Gm/-}$ consisting of triples $(x, y, \alpha)$ such that $x \in \sC^\Gm_{\le 0}$ and $y \in \sD^\Gm_{\le 0}$, then we claim that the inclusion 
\[\sD^\Gm_{F^\Gm/-, \le 0} \to \sD^\Gm_{F^\Gm/-}\]
admits a continuous right adjoint. This right adjoint sends a triple $(x, y, \alpha)$ to the triple $(\tau_{\le 0}x, \tau_{\le 0}y, \alpha_{\le 0})$, where 
\[\alpha_{\le 0}: F\tau_{\le 0}x \to \tau_{\le 0}y\]
is the map obtained by adjunction from 
\[F\tau_{\le 0}x \to Fx \totext{\alpha} y\]
using the fact that $F$ sends $\sC^\Gm_{\le 0}$ to $\sD^\Gm_{\le 0}$. Note that the continuity of this right adjoint can be checked after applying the functors (\ref{eq:projectors}), where it follows from the continuity of the weight truncation functors for $\sC$ and $\sD$. 

\medskip

This choice of $\sD^\Gm_{F^\Gm/-, \le 0}$ defines an $\A^1$ action on the de-equivariantization of $\sD^\Gm_{F^\Gm/-}$, and it has the feature that the forgetful functors to $\sC$ and $\sD$ are $\A^1$-equivariant (using the criterion for $\A^1$ equivariance from \Cref{prop:A1-action-char}). Now we note that the projection onto the first factor of (\ref{eq:projectors}) admits a left adjoint, which sends $x \in \sC^\Gm$ to $(x, F^\Gm x, \id_{F^\Gm x})$, and the composite of this left adjoint with the projection onto the second factor identifies with $F^\Gm$. As a composition of left-lax $\A^1$ equivariant morphisms, $F$ acquires the structure of left-lax $\A^1$-equivariance.  
\end{proof}

\subsubsection{} \Cref{prop:left-lax-preserve-le-0} and \Cref{con:preserve-le-0-left-lax} show that the essential image of the forgetful functor 
\[\Hom_{\A^1\textrm{-left-lax}}(\sC, \sD) \to \Hom_{\Gm}(\sC, \sD)\]
consists of the functors sending $\sC^\Gm_{\le n}$ to $\sD^\Gm_{\le n}$. More precisely, if we let 
\[\Hom_{\Gm}(\sC, \sD)_{\le 0}\] 
denote the subcategory of $\Hom_{\Gm}(\sC, \sD)$ consisting of functors that send $\sC^\Gm_{\le n}$ to $\sD^\Gm_{\le n}$, then we have constructed a section of the forgetful functor 
\begin{equation}\label{eq:oblv-A1-left-lax}
\Hom_{\A^1\textrm{-left-lax}}(\sC, \sD) \to \Hom_{\Gm}(\sC, \sD)_{\le 0}.    
\end{equation}
We expect that (\ref{eq:oblv-A1-left-lax}) is actually an equivalence. We are unable to show this because we can't show that if $F: \sC \to \sD$ is already left-lax $\A^1$-equivariant, then the left-lax $\A^1$-equivariance structure on $F$ constructed in \Cref{con:preserve-le-0-left-lax} agrees with the preexisting lax equivariance structure on $F$. 

\medskip

Let's note that if one could show \textit{a priori} that (\ref{eq:oblv-A1-left-lax}) were fully faithful, then it would automatically be an equivalence. Alternatively, one could show that left-lax equivariance structures on $F$ upgrading a given $\Gm$-equivariance data are equivalent to $\A^1$-actions on the comma category $\sD_{F/-}$ which upgrade the given $\Gm$ action and have the property that the projections onto $\sC$ and $\sD$ are $\A^1$-equivariant. In fact, we expect that this should follow from a homotopy coherent version of \cite{BOURKE2014708}. (This paper was the inspiration for our \Cref{con:preserve-le-0-left-lax}.)

\subsubsection{} A similar argument to the above shows that right-lax $\A^1$-equivariant functors send $\sC^\Gm_{\ge n}$ to $\sD^\Gm_{\ge n}$, and every functor with this property admits the structure of right-lax $\A^1$-equivariance. It follows that every functor $\sC \to \sD$ which is both left-lax $\A^1$-equivariant and right-lax $\A^1$-equivariant is actually $\A^1$-equivariant.  

\subsection{Monads}\label{ss:3-monad}

\subsubsection{}\label{sss:monad-A1-action-constr} Suppose that $\sC$ is a category acted on by $\A^1$, and suppose that $T$ is a $\Gm$-equivariant monad on $\sC$. Then the functor 
\begin{equation}\label{eq:monad-le-functor}
\sC^\Gm_{\le 0} \to \sC^\Gm \totext{\ind} T^\Gm\tmod(\sC^\Gm)    
\end{equation}
admits a continuous right adjoint, and so factors uniquely as 
\[\sC^\Gm_{\le 0} \to T^\Gm\tmod(\sC^\Gm)_{\le 0} \mto T^\Gm\tmod(\sC^\Gm)\]
where the first arrow is generating, the second arrow is fully-faithful, and both arrows have continuous right adjoints. More explicitly, $T^\Gm\tmod(\sC^\Gm)_{\le 0}$ is the full subcategory of $T^\Gm\tmod(\sC^\Gm)$ generated by the image of (\ref{eq:monad-le-functor}). 

\medskip

We claim that under the equivalence of \Cref{prop:A1-action-char}, $T^\Gm\tmod(\sC^\Gm)_{\le 0}$ specifies an $\A^1$-action on $T\tmod(\sC)$. Since the inclusion of this subcategory into $T^\Gm\tmod(\sC^\Gm)$ has a continuous right adjoint by construction, it just remains to see that 
\[T^\Gm\tmod(\sC^\Gm) \simeq \colim_n T^\Gm\tmod(\sC^\Gm)_{\le n}.\]
But this follows from the fact that $\ind: \sC^\Gm \to T^\Gm\tmod(\sC^\Gm)$ is generating and $\sC^\Gm$ is generated by $\sC^\Gm_{\le n}$. 

\medskip 

For this $\A^1$-action on $T\tmod(\sC)$, the functor $\ind: \sC \to T\tmod(\sC)$ is tautologically left-lax $\A^1$-equivariant (and hence $\oblv: T\tmod(\sC) \to \sC$ is right-lax $\A^1$-equivariant). In some sense, this $\A^1$-action is the minimal action with this property. 

\begin{proposition}\label{prop:left-lax-equiv-monad}
In the setup of the preceding paragraph, suppose that $T$ is left-lax $\A^1$-equivariant. Then the $\A^1$-action specified above is the unique action on $T\tmod(\sC)$ such that 
\[\oblv: T\tmod(\sC) \to \sC\] 
is $\A^1$-equivariant. For this $\A^1$-action, the category of invariants $T\tmod(\sC)^{\A^1}$ is monadic over $\sC^{\A^1}$, and the monad is $\tau_{\ge 0}T^\Gm$. 
\end{proposition}

\begin{proof}
The uniqueness of such an $\A^1$ action is clear: for $\oblv$ to be left-lax $\A^1$-equivariant, it must be the case that 
\[T^\Gm\tmod(\sC^\Gm)_{\le 0} \subseteq T^\Gm\tmod(\sC^\Gm) \underset{\sC^\Gm}{\times} \sC^\Gm_{\le 0}.\]
Conversely, if $x \in T^\Gm\tmod(\sC^\Gm)$ has the proerty that $\oblv(x) \in \sC^\Gm_{\le 0}$, then 
\[\oblv(\tau_{> 0}x) = \tau_{> 0}\oblv(x) = 0\]
so $\tau_{> 0 }x = 0$. 

\medskip 

To see that the $\A^1$-action specified above has the property that $\oblv$ is $\A^1$-equivariant, first note that by construction $\ind: \sC \to T\tmod(\sC)$ is left-lax $\A^1$-equivariant, and hence $\oblv: T\tmod(\sC) \to \sC$ is right-lax $\A^1$-equivariant (this doesn't use the assumption that $T$ is left-lax $\A^1$-equivariant). It remains to check that $\oblv$ is left-lax $\A^1$-equivariant. We can check on generators whether the image of $T^\Gm\tmod(\sC^\Gm)_{\le 0}$ under $\oblv$ is contained in $\sC^\Gm_{\le 0}$, and this follows by the assumption that $T$ is left-lax $\A^1$-equivariant. 

\medskip 

It remains to characterize the category of invariants for this $\A^1$-action. By our assumption on $T$, we have a commutative diagram  
\[
\begin{tikzcd}
T^\Gm\tmod(\sC^\Gm)_{\le -1} \arrow[d, shift left=1] \arrow[r] & T^\Gm\tmod(\sC^\Gm)_{\le 0} \arrow[d, shift left=1] \\
\sC^\Gm_{\le -1} \arrow[r] \arrow[u, shift left=1] & \sC^\Gm_{\le 0} \arrow[u, shift left=1]
\end{tikzcd}
\]
which specifies an adjunction in $\dgCat^{\Delta^1}$. Applying the 2-functor 
\[\cof(-): \dgCat^{\Delta^1} \to \dgCat\]
we obtain an adjunction 
\[\ind: \sC^{\A^1} \rightleftarrows T\tmod(\sC)^{\A^1}: \oblv\]
which fits into the commutative diagram 
\[
\begin{tikzcd}
T^\Gm\tmod(\sC^\Gm)_{\le -1} \arrow[d, shift left=1] \arrow[r] & T^\Gm\tmod(\sC^\Gm)_{\le 0} \arrow[d, shift left=1] \arrow[r] & T\tmod(\sC)^{\A^1} \arrow[d, shift left=1]  \\
\sC^\Gm_{\le -1} \arrow[r] \arrow[u, shift left=1] & \sC^\Gm_{\le 0} \arrow[u, shift left=1] \arrow[r] & \sC^{\A^1} \arrow[u, shift left=1] 
\end{tikzcd}
\]
It is easy to see that the functor we have called $\ind$ is generating by precomposing with the quotient map $\sC^\Gm_{\le 0} \to \sC^{\A^1}$, so it follows that $T\tmod(\sC)^{\A^1}$ is monadic over $\sC^{\A^1}$. By construction, the monad $\oblv \circ \ind$ is the monad on the quotient $\sC^\Gm_{\le 0}/\sC^\Gm_{\le -1}$ induced by $T^\Gm$, which shows that it is $\tau_{\ge 0}T^\Gm$. 
\end{proof}

\section{Analysis of the Hecke category} 

\subsubsection{} The goal of this section is to construct the semiorthogonal decomposition of $\sH$ by cells, and show that the associated graded pieces admit $\A^1$-actions whose categories of $\A^1$-invariants recover the asymptotic category $\sJ_c$. 

\medskip 

The outline of this section is as follows:
\begin{itemize}
    \item In \S\ref{ss:4-so-decomp} we construct our desired semiorthogonal decomposition of $\sH$ by two-sided cells. 
    \item In \S\ref{ss:action-of-A1} we show that the theory of weights equips the associated graded pieces $\sH_c$ of the semiorthogonal decompositions with an $\A^1$-action, and that this action is ``left-compatible'' with the monoidal structure on $\sH_c$. 
    \item In \S\ref{ss:asymptotic-cat} we show that the category $\sH_c^{\A^1}$ identifies with the asymptotic category $\sJ_c$. 
    \item In \S\ref{ss:accessible} we study a variant of the subquotient category which will be useful in later sections. 
\end{itemize}

\subsection{A semiorthogonal decomposition indexed by cells}\label{ss:4-so-decomp}

\subsubsection{} Let $\sO_0$ denote $\Sh(N\bs G/B)$. This is the $\ell$-adic version of the unbounded derived category of the ind-completion of the principal block of the BGG category $\sO$. For $w \in W$, let $L_w$ denote the IC sheaf of the Schubert variety $X_w$ and let $P_w$ denote the indecomposable projective cover of $L_w$. 

\medskip

For $S$ a subset of $W$, let $(\sO_0)_S$ denote the full subcategory of $\sO_0$ consisting of objects with the property that the composition factors of all of their perverse cohomology groups are of the form $L_w$ for $w \in S$. Let $(\sO_0)^S$ denote the full subcategory of $\sO_0$ generated under colimits and shifts by $P_w$ for $w \notin S$. 

\begin{proposition}
The inclusion of $(\sO_0)_S$ admits a left adjoint, the inclusion of $(\sO_0)^S$ admits a continuous right adjoint, and 
\[(\sO_0)^S \rightleftarrows \sO_0 \rightleftarrows (\sO_0)_S\]
is a semiorthogonal decomposition of $\sO_0$. 
\end{proposition}

\begin{proof}
Because $(\sO_0)^S$ is generated by objects which are compact in $\sO_0$, it's clear that the inclusion $(\sO_0)^S \to \sO_0$ admits a continuous right adjoint. It remains to identify the right orthogonal $(\sO_0)^{S, \perp}$ with $(\sO_0)_S$. Because $P_w$ is projective, the mapping spectrum
\[\Hom_{\sO_0}(P_w, \sF)\]
vanishes if and only if 
\[\Hom_{\sO_0^\heartsuit}(P_w, H^p_i(\sF))\]
vanishes for all $i$. Now the result follows from the well-known fact that 
\[\dim \Hom_{\sO_0^\heartsuit}(P_w, x) = [x : L_w].\]
\end{proof}

\subsubsection{Example}\label{ss:sl2-calculation} Let's consider the case when $G = SL_2$ and we take $S = \set{s} \subseteq W$. In this case $(\sO_0)^S$ is the subcategory generated by the big projective $P_1$. This category is equivalent to the category of modules for $R = \End(P_1)$ (this ring is commutative). In terms of this description of $(\sO_0)^S$, the inclusion 
\[(\sO_0)^S \to \sO_0\]
identifies with the left adjoint of Soergel's $\V$-functor. 

\medskip

Explicitly, $R = k[x]/x^2$. Let $k$ denote the $R$ module where $x$ acts by zero. Let's calculate $H_*(\V^L(k))$. Because $\V$ is t-exact, we know that $\V^L$ is right t-exact. We know that 
\[k \simeq \coker(R \totext{\cdot x} R) \text{ in } R\tmod^\heartsuit\]  
so it follows that 
\[H_0(\V^L(k)) \simeq \coker(P_1 \totext{x} P_1) \simeq \nabla_s.\]
From the exact triangle 
\[k \to R \to k\]
in $R\tmod$ we obtain an exact triangle 
\[\V^L(k) \to P_1 \to \V^L(k)\]
in $\sO_0$. In particular, we have an exact sequence 
\[0 = H_1(P_1) \to H_1(\V^L(k)) \to \nabla_s \to P_1 \to \nabla_s \to 0\]
so we calculate that $H_1(\V^L(k)) \simeq L_s$. For all $i \ge 1$ we have an exact sequence 
\[0 = H_{i + 1}(P_1) \to H_{i + 1}(\V^L(k)) \to H_i(\V^L(k)) \to H_i(P_1) = 0,\]
so it follows that 
\[
H_i(\V^L(k)) \simeq \begin{cases}
L_s & i > 0 \\
\nabla_s & i = 0 \\
0 & i < 0
\end{cases}
\] 

\medskip 

Using this calculation, one can compute the projections $\tau_S(x)$ and $\tau^S(x)$ for the various indecomposable objects of $\sO_0$. As an illustrative example, let's consider the case $x = L_1$. Then 
\[\tau^S(x) \simeq \V^L\V(x) \simeq \V^L(k)\]
whose homology groups we have already computed. Since 
\[\tau_S(x) \simeq \cof(\tau^S(x) \to x),\]
we learn that
\[
H_i(\tau_S(x)) \simeq \begin{cases}
L_s & i \ge 1 \\
0 & i < 1
\end{cases}
\]   

\subsubsection{}\label{ss:defn-H} Let $\sH$ denote the equivariant Hecke category $\Sh(B\bs G /B)$. In \cite[Theorem 1.8.1]{benzvi2015charactertheorycomplexgroup} it was shown that $\sH$ is a pivotal semi-rigid monoidal category. Let 
\[\pi: N\bs G/B \to B\bs G/B\]
denote the natural projection. The IC sheaf $L_w$ has the property of being $T$-equivariant (as an object of the abelian category $\sO_0^\heartsuit$), and we will abuse notation by using the same symbol $L_w$ to denote the corresponding object $L_w \in \Sh(B\bs G/B)$. 

\medskip

Like in the case of $\sO_0$, given a subset $S \subseteq W$, define $\sH_S$ to be the subcategory of $\sH$ consisting of objects, all of whose perverse cohomologies are filtered colimits of extensions of $L_w$ for $w \in S$, and define $\sH^S$ to be the subcategory of $\sH$ generated by $\pi_!P_w$ for $w \notin S$. 

\begin{proposition}\label{prop:semiorthogonal-compat-H-O}
The inclusion $\sH^S \to \sH$ admits a continuous right adjoint, $\sH_S \to \sH$ admits a continuous left adjoint, and 
\[\sH^S \rightleftarrows \sH \rightleftarrows \sH_S\]
is a semiorthogonal decomposition of $\sH$. Additionally, $\pi^!: \sH \to \sO_0$ sends $\sH_S$ to $(\sO_0)_S$ and $\sH^S$ to $(\sO_0)^S$. 
\end{proposition}

\begin{proof}
As before, the inclusion $\sH^S \to \sH$ admits a continuous right adjoint because $\sH^S$ is generated by objects which are compact in $\sH$. The assertion that $\sH^{S, \perp} = \sH_S$ follows from the fact that you can check whether $L_w$ appears as a composition factor of $x \in \sH^\heartsuit$ after applying $\pi^!$. This also shows that $\pi^!(\sH_S) \subseteq (\sO_0)_S$. 

\medskip

It remains to show that $\pi^!(\sH^S) \subseteq (\sO_0)^S$. We need to show that $\pi^!\pi_!P_w \in (\sO_0)^S$ for $w \notin S$. It suffices to show that 
\[\Hom(\pi^!\pi_!P_w, \sF) = 0\]
for all $\sF \in (\sO_0)_S$. Because $\pi^!\pi_!P_w$ is compact, we can assume that $\sF = L_v$ for $v \in S$. By adjunction
\[\Hom(\pi^!\pi_!P_w, L_v) = \Hom(\pi^*\pi_!P_w, L_v)[-2\dim T] = \Hom(P_w, \pi^!\pi_* L_v)[-2\dim T].\]
Since 
\[\pi^!\pi_* L_v \simeq C_*^{BM}(T) \otimes L_v\]
the result follows. 
\end{proof}

\subsubsection{}\label{sss:cells-defn} Recall the notion of two-sided cells (already introduced in \S\ref{sss:intro-cells}): if $\sH_w^\access$ refers to the smallest two-sided ideal of $\sH$ containing $L_w$ and stable under colimits and shifts, then we say that $v \le_{LR} w$ if $L_v \in \sH_w^\access$. This defines a preorder on $W$ and the two-sided cells of $W$ are the corresponding partition: $w$ and $v$ lie in the same two-sided cell if $w \le_{LR} v$ and $v \le_{LR} w$. 

\medskip 

Let $c$ denote a two sided-cell in $W$. Let $W_{\le c}$, $W_{< c}$ denote the subsets of $w \in W$ such that $w \le_{LR} c$ and $w <_{LR} c$, respectively. Denote the corresponding localizations of $\sH$ by 
\[\sH_{\le c} \text{ and } \sH_{< c},\]
respectively. 

\begin{proposition}\label{prop:cells-ideals}
$\sH_{\le c}$ and $\sH_{< c}$ are monoidal two-sided ideals in $\sH$. 
\end{proposition}

\begin{proof}
If $S$ is a subset of $W$, then because $\sH_S$ is stable under colimits and $\sH$ is generated by $L_w$ for $w \in W$, one can show that $\sH_S$ is a two-sided ideal in $\sH$ just by showing that $L_w * (-)$ and $(-) * L_w$ preserve $\sH_S$. Every object $x \in \sH_S$ can be written as the limit of its truncations for the t-structure
\[x = \lim_n \tau_{\le n}x,\]
and each $\tau_{\le n}x$ can be written as a colimit of $L_v$ with $v \in S$. Because $L_w$ is left- and right-dualizable and $\sH_S$ is stable under limits as well as colimits, it suffices to show that $L_w * L_v$ and $L_v * L_w$ lie in $\sH_S$ for all $v \in S$. This is a direct consequence of the definition of two-sided cells (see \S\ref{sss:cells-defn}). 
\end{proof}

\subsubsection{}\label{ss:defn-Hc} It follows from \Cref{prop:semiorthogonal-compat-H-O}, \Cref{prop:cells-ideals}, \Cref{prop:semiorthogonal-semirigid}, and \Cref{prop:semirigid-loc-coloc} that we obtain a system of semiorthogonal decompositions 
\begin{equation}\label{eq:semiorthogonal-decomp-cells}
\sH_{\not\le c} \rightleftarrows \sH \rightleftarrows \sH_{\le c}    
\end{equation}
where the subcategories of $\sH$ admit monoidal structures making them semi-rigid monoidal categories, and the localization and colocalization functors are monoidal.

\medskip 

If we let $\sH_c$ denote the left-orthogonal to $\sH_{< c}$ inside of $\sH_{\le c}$, then it follows from \Cref{prop:semiorthogonal-semirigid} and \Cref{prop:semirigid-loc-coloc} that $\sH_c$ admits a unique monoidal structure such that the projection $\sH \to \sH_c$ is monoidal, and with respect to this monoidal structure $\sH_c$ has the structure of a semi-rigid category. We also note that 
\[\sH_c \simeq \sH_{\le c} \underset{\sH}{\otimes} \sH_{\not < c}\]
so we may think of the categories $\sH_c$ as the associated-graded pieces of the system of semiorthogonal decompositions (\ref{eq:semiorthogonal-decomp-cells}). 

\subsubsection{Example}\label{sss:sl2-1-qcoh} Let's return to the case of $SL_2$. Let $c$ denote the cell $\set{1} \subseteq W$. Then $\sH_c$ is the subcategory of $\sH$ generated by $\pi_!P_1$. 
\[\pi_i\End(\pi_!P_1) = \begin{cases}
k & i = 0, 1 \\
0 & \text{otherwise}
\end{cases}\] 
so it follows that $\sH_c$ identifies with $\QCoh(\Omega \A^1)$ as a plain dg category. One can show that the monoidal structure on $\sH_c$ identifies with the monoidal structure on $\QCoh(\Omega \A^1)$ coming from the group structure on $\Omega \A^1$. 

\subsection{An action of the monoid $\A^1$}\label{ss:action-of-A1}

\subsubsection{}\label{ss:defn-W} In \cite{gaitsgory2026excursionalgebra}, we introduced a sheaf theory $\Sh(-)^\relev$ which is defined on schemes over $\F_q$. If $X_0$ is a scheme over $\F_q$ whose base-change to $\ol{\F}_q$ is $X$, then in keeping with the convention of \textit{loc. cit.} we somewhat abusively denote the value of $\Sh(-)^\relev$ on $X_0$ by $\Sh(X)^\relev$. The definition of $\Sh(X)^\relev$ is as follows: the given $\F_q$ structure on $X$ equips this scheme with a geometric Frobenius endomorphism $\Frob$, and 
\[\Sh(X)^\relev \subseteq \Sh(X)\] 
is the subcategory generated under colimits and shifts by irreducible perverse sheaves $\sF \in \Perv(X)$ such that for some $n \ge 1$ there exists an isomorphism $\sF \simeq (\Frob^n)^*\sF$. 

\medskip

The important feature of $\Sh(X)^\relev$ is that it admits an action of a pro-algebraic group which we will denote $\W$.\footnote{This is the group which is denoted $\Z^{\mathrm{alg, wt}}$ in \textit{loc. cit.}} This action is constructed using the theory of weights for $\ell$-adic sheaves on varieties over finite fields and it has the property that $(\Sh(X)^\relev)^\W$ is the category $\Sh(X)^{\Weil, \mathrm{wt}}$ (introduced in \textit{loc. cit.}) which is a minor variant of the category of mixed sheaves (sometimes we will call an object of $\Sh(X)^{\Weil, \mathrm{wt}}$ a \textit{strongly} mixed sheaf). If $\Weil_q \subseteq \ol{\Q}^\times$ denotes the subgroup of $q$-Weil numbers, then $\W$ is the Cartier dual of $\Gm^\wedge_{\Weil_q}$. The homomorphism 
\[\mathrm{wt}: \Gm^\wedge_{\Weil_q} \to \Z\]
which sends a connected component of $\Gm^\wedge_{\Weil_q}$ to the weight of its underlying Weil number gives rise to a distinguished subgroup $\Gm \subseteq \W$ after applying Cartier duality. The inclusion 
\[\Gm^\wedge_{\Weil_q} \subseteq \Gm\]
gives rise to a homomorphism $\Frob: \Z \to \W$ whose image is Zariski dense. One of the defining features of the action of $\W$ on $\Sh(X)^\relev$ is that the element $\Frob \in \W$ acts by pullback along geometric Frobenius. 

\subsubsection{}\label{ss:A1-action} Another feature of the action of $\W$ on $\Sh(X)^\relev$ is that the action of the weight $\Gm$ extends to an action of the commutative monoid scheme $\A^1$. In terms of \Cref{prop:A1-action-char} the category $(\Sh(X)^\relev)^\Gm$ is a minor variant of the category of graded mixed sheaves introduced in \cite{Ho2022RevisitingMG}, and the subcategory $(\Sh(X)^\relev)^\Gm_{\le 0}$ is the category generated under colimits and shifts by perverse sheaves which are weight $\le 0$ in the sense of \cite{BBD} (this category was considered previously in \cite{morel2008complexes}). 

\subsubsection{Warning}\label{ss:warning-mixed-wt-zero} The category $(\Sh(X)^\relev)^\Gm_{\le 0}$ is not the derived category of mixed sheaves of weight $\le 0$ in the sense of \cite{BBD}. Unlike the category defined in \textit{loc. cit.}, the category $(\Sh(X)^\relev)^\Gm_{\le 0}$ is stable under shifts. This is advantageous for certain purposes, but it has the following drawback: if $f$ is a smooth proper map, then $f_*$ does not preserve weights in our sense.

\subsubsection{}\label{ss:A1-action-defn} Let $f: X \to Y$ be a smooth morphism between schemes. Define the perverse pullback functor 
\[f^{!*}: \Sh(Y)^\relev \to \Sh(X)^\relev\]
by 
\[f^{!*} = f^*[\dim(X/Y)] = f^![-\dim(X/Y)].\]
By construction, this functor is t-exact for the perverse t-structure. If we choose a square-root of $q$, we can equip $f^{1/2}$ with the structure of equivariance for $\W$ as follows: consider the functor (which we will also denote $f^{1/2}$) from $\Sh(Y)^{\Weil, \mathrm{wt}}$ to $\Sh(X)^{\Weil, \mathrm{wt}}$ defined by 
\[f^{!*} = f^*(\tfrac{\dim(X/Y)}{2})[\dim(X/Y)] = f^!(-\tfrac{\dim(X/Y)}{2})[-\dim(X/Y)].\]
Then the functor between the categories of relevant sheaves arises as the de-equivariantization of this functor between the categories of strongly mixed sheaves, equipping it with equivariance data for $\W$. It is easy to see that the functor $f^{!*}$ is $\A^1$-equivariant for the $\A^1$-action of \S\ref{ss:A1-action}.

\medskip

If $\sY$ is an Artin stack, then define 
\begin{equation}\label{eq:sh-relev-defn}
\Sh(\sY)^\relev = \lim_{Y \to \sY} \Sh(Y)^\relev    
\end{equation}
where the limit is taken over all schemes $Y$ mapping smoothly to $\sY$ and the transition maps are the perverse pullback functors. In other words, $\Sh(\sY)^\relev$ is the full subcategory of $\Sh(\sY)$ consisting of sheaves whose pullback to some (equiv. any) smooth cover is relevant. Taking the limit (\ref{eq:sh-relev-defn}) internally to categories acted on by $\A^1$ defines a natural $\A^1$-action on $\Sh(\sY)^\relev$. This $\A^1$-action has the property that if $f: Y \to \sY$ is a smooth cover, then an object $\sF$ of $\sY$ is in $\Shv(\sY)^\Gm_{\le 0}$ if and only if its pullback to $Y$ is, and similarly for $\Shv(\sY)^\Gm_{\ge 0}$.

\begin{proposition}
If $\sY = N \bs G /B$ or $\sY = B \bs G / B$, then 
\[\Sh(\sY)^\relev = \Sh(\sY).\]
\end{proposition}

\begin{proof}
These categories are generated by the IC sheaves $L_w$, whose pullbacks to $G/B$ are relevant. 
\end{proof}

\subsubsection{}\label{ss:A1-action-on-subquotient-defn} The above discussion shows that $\sH$ acquires an action of the monoid $\A^1$. Since the action of $\Gm \subseteq \A^1$ is compatible with six functors, the monoidal structure on $\sH$ is compatible with this $\Gm$ action. However, we caution that because of the behavior described in \S\ref{ss:warning-mixed-wt-zero}, the $\A^1$-action on $\sH$ is \textit{not} compatible with its monoidal structure.   

\begin{proposition}\label{prop:right-cat-A1-equiv}
For any $S \subseteq W$, the category $\sH_S$ is stable under the $\A^1$ on $\sH$. 
\end{proposition}

\begin{proof}
The is follows by the arguments of \cite[\S3.1]{gaitsgory2026excursionalgebra}. In a little more detail, to show that $\sH_S$ is stable under the action of $\Gm$, then as in the proof of \cite[Lemma 3.1.4]{gaitsgory2026excursionalgebra} it suffices to show that the comonad 
\[\sH \totext{\Av_\Gm} \sH^\Gm \totext{\oblv} \sH\]
preserves $\sH_S$. The comonad is t-exact (this can be checked after applying $\oblv$ to $\sO_0$, where it follows from the corresponding statement for schemes, which is the content of \cite[\S2.6]{gaitsgory2026excursionalgebra}), so it's enough to check on objects in the heart. Since the comonad commutes with colimits, it suffices to show that it sends $L_w$ to an object of $\sH_S$ for all $w \in S$. Now each $L_w$ admits the structure of a strongly mixed sheaf, so the value of the comonad on $L_w$ is simply $\sO_\Gm \otimes L_w \in \sH_S$.

\medskip 

Once we know that $\sH_S$ is stable under the action of $\Gm$, as in \cite[Proposition 3.1.7]{gaitsgory2026excursionalgebra} to prove that it is stable under $\A^1$ it suffices to show that $\sH_S^\Gm$ is preserved by the weight truncation functors on $\sH^\Gm$. The weight truncation functors are t-exact (again this follows by applying $\oblv$ to $\sO_0$), so it suffices to check this for objects of the heart. Now the result follows because 
\[\tau_{\le i}L_w = \begin{cases}
L_w & i \ge 0 \\
0 & i < 0
\end{cases}\]
\end{proof}

By \Cref{prop:right-cat-A1-equiv} the inclusions 
\[\sH_{< c} \to \sH_{\le c} \to \sH\] 
are $\A^1$-equivariant. The left adjoints of these inclusions are left-lax $\A^1$-equivariant, but we caution that the left adjoints are \textit{not} strictly $\A^1$-equivariant. By the compatibility of the $\A^1$ action on $\sH$ with smooth pullbacks (see \S\ref{ss:A1-action-defn}) and \Cref{prop:semiorthogonal-compat-H-O} we see that the forgetful functor $\sH_{\le c} \to (\sO_0)_{\le c}$ is $\A^1$-equivariant. By writing 
\[\sH_c \simeq \sH_{\le c}/\sH_{< c}\]
$\sH_c$ acquires an action of $\A^1$ such that the projection $\sH_{\le c} \to \sH_c$ is tautologically $\A^1$-equivariant. We caution that the inclusion $\sH_c \to \sH_{\le c}$ is only left-lax $\A^1$-equivariant.

\subsubsection{Example} If we set $G = SL_2$ and take $c = \set{1}$, then like in \S\ref{ss:sl2-calculation}, one can calculate that
\[
H_i(\tau_{< c}(L_1)) \simeq \begin{cases}
L_s\bracket{-2i + 1} & i \ge 1 \\
0 & i < 1
\end{cases}
\]
and 
\[
H_i(\tau_c(L_1)) \simeq \begin{cases}
L_s\bracket{-2i - 1} & i > 0 \\
\nabla_s\bracket{-1} & i = 0 \\
0 & i < 0
\end{cases}
\] 
We note that for the action of $\A^1$ on $\sH_c$ defined above, the object $\tau_c(L_1)$ is pure of weight zero (since it is the image of a sheaf which is pure of weight zero under $\tau_c$). However, the image of $\tau_c(L_1)$ in $\sH$ is evidently not pure of weight zero. 

\subsubsection{Warning}\label{sss:automorphisms} Suppose that $\sC$ is a category acted on by $\A^1$. By \Cref{prop:A1-action-char}, the action of $\A^1$ on $\sC$ is completely determined by the action of $\Gm$ plus the data of a certain full subcategory $\sC^\Gm_{\le 0} \subseteq \sC^\Gm$. However, it is important to note that it may be the case that $\sC$ has more automorphisms as a category acted on by $\Gm$ than it does as a category acted on by $\A^1$. This leads to a somewhat counterintuitive situation, where different choices of $\sC^\Gm_{\le 0}$ can give rise to equivalent $\A^1$ actions on $\sC$. When this happens, the $\A^1$-actions are equivalent via an isomorphism that induces a non-trivial automorphism of $\sC^\Gm$, which interchanges the two choices of $\sC^\Gm_{\le 0}$. 

\subsubsection{}\label{sss:left-compat} If $\sA$ is a monoidal category with an action of $\A^1$, such that the monoidal structure is endowed with $\Gm$-equivariance data, we will say that the monoidal structure of $\sA$ is left-compatible with the action of $\A^1$ on $\sA$ if the maps 
\begin{align*}
&\mult: \sA \otimes \sA \to \sA \\
&\unit: \Vect \to \sA
\end{align*}
have the property that the associated maps 
\begin{align*}
\mult^\Gm: \sA^\Gm_{\le 0} \otimes \sA^\Gm_{\le 0} \to \sA^\Gm \\
\unit^\Gm: \Rep(\Gm)_{\le 0} \to \sA^\Gm
\end{align*}
factor through $\sA^\Gm_{\le 0}$.

\medskip 

Suppose that $\sA$ is a monoidal category equipped with an action of $\A^1$, such that the monoidal structure is endowed with $\Gm$-equivariance data. Using \Cref{prop:A1-action-char}, for any integer $k$ we can define a new category $\sA(k)$ with an action of $\A^1$, with $\Gm$-equivariant monoidal structure. Namely, $\sA(k) = \sA$ as a monoidal category acted on by $\Gm$, but $\sA(k)^\Gm_{\le n} = \sA^\Gm_{\le n - k}$. If we forget the monoidal structure on $\sA$, then $\sA(k)$ and $\sA$ are equivalent as plain categories acted on by $\A^1$ (acting by $\sO(k)$ is a $\Rep(\Gm)$-linear automorphism of $\sA^\Gm$ that intertwines the two $\le 0$ categories, as in \S\ref{sss:automorphisms}). However, $\sA(k)$ will not typically be equivalent to $\sA$ as a category acted on by $\A^1$ with a $\Gm$-equivariant monoidal structure; in fact, it may be the case that the action of $\A^1$ on $\sA(k)$ is left-compatible with the monoidal structure, even when the action on $\sA$ is not. 

\subsubsection{} The monoidal structure of $\sH_c$ is not left-compatible with the action of $\A^1$ as-is. However, let $a$ denote the value of Lusztig's $a$-function when evaluated on our given two-sided cell $c$.

\begin{proposition}\label{prop:shifted-A1-action-left-compatible}
The monoidal structure on $\sH_c(a)$ is left-compatible with the natural action of $\A^1$. 
\end{proposition}

\begin{proof}
Let's first show that the unit of $\sH_c^\Gm$ is concentrated in weights $\le -a$. Because the projection $\sH_{\le c} \to \sH_c$ is $\A^1$-equivariant and monoidal, it suffices to check that the unit of $\sH_{\le c}$ is concentrated in weights $\le -a$. By the definition of the $\A^1$-action on $\sH_{\le c}$ (see \S\ref{ss:A1-action-on-subquotient-defn}), we can check this after forgetting to $(\sO_0)_{\le c}$. We have an exact triangle 
\[\tau_{\le -a} 1_{\le c} \to 1_{\le c} \to \tau_{> -a} 1_{\le c}\]
coming from the semiorthogonal decomposition of $(\sO_0)_{\le c}^{\Gm}$ according to weights. It follows formally from the theory of semiorthogonal decompositions that the first map in this triangle is an isomorphism if and only if the second map is zero. Since 
\[\Hom_{(\sO_0)^\Gm_{\le c}}(1_{\le c}, \tau_{> -a} 1_{\le c}) \simeq \Hom_{(\sO_0)^\Gm}(1, \tau_{> -a} 1_{\le c})\] 
it suffices to show that the latter mapping space is zero. If we let $i$ denote the closed embedding $*/N \to N \bs G / B$, then this mapping space identifies with 
\begin{equation}\label{eq:hom-space-unit-subquot}
\Hom_{\Sh(*/N)^\Gm}(\triv, i^!\tau_{> - a}1_{\le c}).    
\end{equation}
However, we claim that for any object $x$ of $\sH_{\le c}^\Gm$ with weights strictly larger than $-a$, the !-fiber $i^!x$ has weights strictly larger than zero (this immediately implies the vanishing of (\ref{eq:hom-space-unit-subquot})). Since any such $x$ can be written as a limit of colimits of copies of $L_w\bracket{j}$, with $w \le c$ and $j > -a$, we can reduce to the case of $L_w\bracket{j}$. Now let $P_{w, v}$ denote the Kazhdan Lusztig polynomials. Define $p_{w, n}$ to be the coefficients of $P_{1, w}$:
\[P_{1, w}(q) = \sum p_{w, n}q^n.\]
By Kazhdan-Lusztig theory, we know that $i^!L_w$ is isomorphic to 
\[\bigoplus_{n \in \Z} \triv^{\oplus p_{w, n}}[2n - \ell(w)](n - \ell(w)/2).\] 
If we let $\delta(w)$ denote the degree of $P_{1, w}$, then by \cite[Equation 1.3.(a)]{LUSZTIG1987536} and \cite[Theorem 5.4]{lusztig1985cells} (see also \cite[Theorem 6.8.(i)]{Curtis1988} and \cite[Lemma 6.9.(ii)]{Curtis1988}) we know that for any $w \le c$, 
\[2\delta(w) - \ell(w) \le -a(w).\]
It follows that $i^!L_w\bracket{j}$ lies in weights strictly greater than zero in $\Sh(*/N)^\Gm$ for any $w \le c$ and any $j > -a$, as desired. 

\medskip

We now show that the product of two objects of $\sH_c^\Gm$ concentrated in weights $\le - a$ remains in weights $\le - a$. Since the inclusion $\sH_c \to \sH_{\le c}$ is left-lax $\A^1$ equivariant, if $x$ and $y$ are in weights $\le - a$ as objects of $\sH_c^\Gm$, then they will be in weights $\le -a$ as objects of $\sH_{\le c}^\Gm$, as well. Since the projection $\sH_{\le c} \to \sH_c$ is $\A^1$-equivariant, it suffices to show that for any $x, y \in \sH_{\le c}^\Gm$ with weights $\le 0$, the truncation $\tau_{> a}(x * y)$ lies in $\sH_{< c}^\Gm$. We will do this in 5 steps.

\medskip

\textit{Step 1}. Suppose that $x = L_w$ and $y = L_v$ for some $w, v \le c$. If either $x$ or $y$ is $< c$, then $x * y$ is $< c$. If both $x$ and $y$ lie in $c$, then by purity $\tau_{> a}(x * y)$ lies in cohomological degree $>a$. Now the result follows from \cite[\S2.3]{LUSZTIG199785}. 

\medskip

\textit{Step 2}. If both $x$ and $y$ are constructible, then the result follows from Step 1 because the condition that $\tau_{> a}(x * y)$ lies in $\sH_{< c}^\Gm$ is stable under the formation of finite colimits. 

\medskip 

\textit{Step 3}. Now suppose that $x$ is constructible and $y$ is compact. To show that $\tau_{> a}(x * y)$ lies in $\sH_{< c}^\Gm$, it's enough to check that all of its perverse cohomologies lie in $\sH_{< c}^\Gm$. By \cite[Proposition 2.9.2]{gaitsgory2026excursionalgebra}, $\tau_{> a}$ is t-exact. Additionally, $x * (-)$ has finite cohomological amplitude, and because $x$ is dualizable it commutes with limits. Since every bounded below truncation $\tau^p_{\le n}y$ of $y$ for the perverse t-structure is constructible, the result follows from Step 2. 

\medskip 

\textit{Step 4}. Suppose that both $x$ and $y$ are compact. Then $(-) * y$ has finite cohomological amplitude and because $\sH_{\le c}$ is semi-rigid, $y$ is dualizable. Then arguing like in Step 3 with the roles of $x$ and $y$ reversed, we may replace $x$ by its eventually coconnective truncations, which reduces to Step 3. 

\medskip 

\textit{Step 5}. If $x$ and $y$ are general objects of $\sH_{\le c}$, then the result follows from Step 4 using the compact generation of $\sH_{\le c}$ and the compatibility of the monoidal structure with colimits.  
\end{proof}

\subsection{The asymptotic category}\label{ss:asymptotic-cat}

\begin{construction}\label{constr:monoidal-structure-A1-invts}
Suppose that $\sA$ is a monoidal category with an action of $\A^1$ and suppose that the monoidal structure on $\sA$ is left-compatible with the action of $\A^1$. Then the category $\sA^{\A^1}$ acquires a natural $\A^1$-action. 
\end{construction}

\begin{proof}
The subcategory $\sA^{\Gm}_{\le 0}$ is evidently a monoidal subcategory of $\sA^\Gm$. In the semiorthogonal decomposition 
\[\sA^\Gm_{\le -1} \rightleftarrows \sA^\Gm_{\le 0} \rightleftarrows \sA^{\A^1}\]
the subcategory $\sA^\Gm_{\le -1}$ is a two-sided ideal inside of $\sA^\Gm_{\le 0}$, so like in the proof of \Cref{prop:semiorthogonal-semirigid}, \cite[Proposition 2.2.1.9]{lurie2017higher} equips $\sA^{\A^1}$ with a unique monoidal structure such that the projection $\sA^\Gm_{\le 0} \to \sA^{\A^1}$ is monoidal. 
\end{proof}

\subsubsection{} Define $\sJ_c = \sH_c(a)^{\A^1}$. By \Cref{constr:monoidal-structure-A1-invts} this category acquires a natural monoidal structure. The goal of this section is to show that $\sJ_c$ is equivalent as a monoidal category to Lusztig's asymptotic category as defined in \cite[\S2]{LUSZTIG199785} (see also \cite[\S5]{EliasWilliamson2021}).

\begin{proposition}\label{prop:Hlec-semisimple}
$\sH_{\le c}^{\A^1}$ is semisimple, and its simple objects are indexed by $w \in W_{\le c}$. 
\end{proposition}

\begin{proof}
By \Cref{prop:semiorthogonal-compat-H-O}, $\sH_{\le c}$ is monadic over $(\sO_0)_{\le c}$. The monad is left-lax $\A^1$-equivariant (it suffices to check this after including $(\sO_0)_{\le c}$ into $\sO_0$), so \Cref{prop:left-lax-equiv-monad} shows that $\sH_{\le c}$ admits a unique $\A^1$-action such that the forgetful functor to $(\sO_0)_{\le c}$ is $\A^1$-equivariant. Since $\sH_{\le c}$ is monadic over $\sO$, and both the $\A^1$-action from \S\ref{ss:A1-action-on-subquotient-defn} and this new one have the property that $\oblv$ is strictly $\A^1$-equivariant, it follows from the uniqueness statement of \Cref{prop:left-lax-equiv-monad} that the ``new'' $\A^1$ action agrees with the one from \S\ref{ss:A1-action-on-subquotient-defn}. 

\medskip 

It follows from the characterization of the category of $\A^1$-invariants in \Cref{prop:left-lax-equiv-monad} that $\sH_{\le c}^{\A^1}$ is monadic over $(\sO_0)_{\le c}^{\A^1}$. In fact, because $C_*(T)^{\A^1} = k$, we know that 
\[\sH_{\le c}^{\A^1} \simeq (\sO_0)_{\le c}^{\A^1}.\]
Now the result follows from \cite[Proposition 2.7.5]{gaitsgory2026excursionalgebra}.
\end{proof}

\begin{proposition}\label{prop:Jc-char}
$\sJ_c$ is semisimple, and its simple objects are indexed by $w \in c$. 
\end{proposition}

\begin{proof}
Suppose that $x \in \sH_c^{\A^1}$. Then $x$ is left-lax $\A^1$-equivariant as an object of $\sH_{\le c}(a)$ (for the rest of this proposition we will consider the $a$-shifted $\A^1$-action on $\sH_{\le c}$). Applying the projection onto $\sH_c$ to the exact triangle 
\[\tau_{\le -1}x \to x \to \gr_0(x)\]
we see that $\tau_{\le -1}x$ lies in $\sH_{< c}$ and that $\tau_{\not<c}\gr_0(x) = x$ (here $\tau_{\le -1}$ and $\gr_0$ refer to the weight filtration). In particular, every $x \in \sH_c^{\A^1}$ arises as $\tau_{\not<c}y$ for $y \in \sH_{\le c}^{\A^1}$. By \Cref{prop:Hlec-semisimple}, $\sH_{\le c}^{\A^1}$ is semisimple, with simple objects $L_w$ for $w \in W_{\le c}$. Since $\tau_{\not < c}L_w = 0$ when $w < c$, it remains to show that 
\[\Hom(\tau_{\not < c}L_w, \tau_{\not < c}L_v) = k^{\delta_{wv}}\]
for $w, v \in c$. To see this, note that
\[\Hom_{\sH_c}(\tau_{\not < c}L_w, \tau_{\not < c}L_v) \simeq \Hom_{\sH_{\le c}}(\tau_{\not < c}L_w, L_v) \simeq \Hom_{\sH_{\le c}}(\tau_{\ge 0}\tau_{\not < c}L_w, L_v),\]
so it's enough to show that $\tau_{\ge 0}\tau_{\not < c}L_w \simeq L_w$. Now note that there is a map 
\[\pi_!P_w \to L_w\]
in $\sH$ which induces an isomorphism after applying $\tau_{\ge 0}$ (the weight truncation for $\sH$). By the left-lax equivariance of the projector $\tau_{\le c}$, this implies that the map 
\[\tau_{\le c}\pi_!P_w \to L_w\]
is also an isomorphism after applying $\tau_{\ge 0}$. By the left-lax equivariance of $\tau_{\not < c}$ (viewed as an endofunctor of $\sH_{\le c}$), this implies that 
\[\tau_{\le c}\pi_!P_w \simeq \tau_{\not < c} \tau_{\le c}\pi_!P_w \to \tau_{\not < c}L_w\]
is an isomorphism after applying $\tau_{\ge 0}$. Therefore
\[\tau_{\ge 0}\tau_{\not < c}L_w \simeq \tau_{\ge 0}\tau_{\le c}\pi_!P_w \simeq \tau_{\ge 0}L_w \simeq L_w,\]
as desired. 
\end{proof}

\subsubsection{} In the context of the asymptotic category, we will write $J_w = \tau_{\not<c} L_w$ for the simple object corresponding to $w$. 

\begin{proposition}
$\sJ_c$ is equivalent to Lustzig's category as a monoidal category. 
\end{proposition}

\begin{proof}
By \cite[Corollary 5.4.4.7]{lurie2017higher}, it suffices to construct an equivalence of non-unital monoidal categories between our $\sJ_c$ and Lusztig's category (indeed, every non-unital equivalence between unital monoidal categories is automatically quasi-unital). Unwinding the definition of the monoidal structure from \Cref{constr:monoidal-structure-A1-invts}, we see that 
\[J_w * J_v = \tau_{\ge a}\tau_{\not<c}L_w * L_v\]
By purity, this agrees with 
\[H^a(\tau_{\not<c}L_w * L_v),\]
which is Lusztig's definition of truncated convolution.
\end{proof}

\subsubsection{} Recall (see \cite[\S6]{Curtis1988}) that an element $d$ in a cell $c$ is said to be a \textit{Duflo involution} if 
\[a(d) = \ell(d) - 2\delta(d)\]
(in \cite[\S1]{LUSZTIG1987536} these elements are called ``distinguished involutions''). This numerical equality implies, among other things, that $d$ has order 2 (see \cite[Proposition 1.4(a)]{LUSZTIG1987536}, \cite[Theorem 6.8(iii)]{Curtis1988}). 

\begin{proposition}\label{prop:jc-unit}
The unit of $\sJ_c$ is 
\[\bigoplus_{d \in \sD_c} J_d[-a]\]
where $\sD_c$ is the set of Duflo involutions in the two-sided cell $c$. 
\end{proposition}

\begin{proof}
This follows immediately from \Cref{prop:Jc-char} and Lusztig's description of the unit of $\sJ_c$ (see \cite{LUSZTIG199785}). However, for completeness, we will give another proof of this, using the the fact that the unit of $\sJ_c$ is $\gr_0(1_c)$. Because the projection $\sH_{\le c} \to \sH_c$ is $\A^1$-equivariant, we need to show that 
\[\tau_{\not < c}\tau_{\ge -a} 1_{\le c} \simeq \bigoplus_{d \in \sD_c} L_d.\]
Arguing as in \Cref{prop:shifted-A1-action-left-compatible}, we see that 
\[\tau_{\not < c}\tau_{\ge -a} 1_{\le c} \simeq \bigoplus_{\substack{w \in c \\ 2\delta(w) - \ell(w) = -a}} L_w^{\oplus p_{w, \delta(w)}} \]
But Duflo involutions are definitionally elements of $c$ such that $\ell(w) - 2\delta(w) = a$, and it is well-known that for $d$ a Duflo involution, $p_{d, \delta(d)} = 1$ (see \cite[Proposition 1.4(a)]{LUSZTIG1987536} or \cite[Theorem 6.8(iii)]{Curtis1988}).
\end{proof}

\subsection{A variant: accessible sheaves}\label{ss:accessible}

\subsubsection{} Let $\sH_{\le c}^\access$ denote the full subcategory of $\sH_{\le c}$ generated under colimits and shifts by the irreducible objects $L_w$ with $w \le c$. 

\begin{proposition}
The category $\sH_{\le c}^\access$ is generated by objects which are compact in $\sH_{\le c}$. (In fact, compact in $\sH$.)
\end{proposition}

\begin{proof}
By \Cref{prop:semiorthogonal-compat-H-O}, we may write
\[L_w \simeq \colim_{\Delta^\op} (\pi_!\pi^!)^\sbullet L_w,\]
where each term in the geometric realization lies in $\sH_{\le c}$. Each term in the geometric realization is compact in $\sH_{\le c}$ by virtue of being compact in $\sH$, and each term is a finite extension of shifts of $L_w$, so it lies in $\sH_{\le c}^\access$.
\end{proof}

\subsubsection{} It follows that we obtain a semiorthogonal decomposition 
\[\sH_{\le c}^\access \rightleftarrows \sH_{\le c} \rightleftarrows \sH_{\le c}^\inaccess\]
where $\sH_{\le c}^\inaccess$ is by definition the right orthogonal of $\sH_{\le c}^\access$ in $\sH_{\le c}$.  

\subsubsection{Example}\label{sss:access-vs-all-sl2} Let's consider the case when $G = SL_2$ and $c = {s} \subseteq W$. For simplicity, we will consider the category $(\sO_0)_{\le c}^\access$, which is defined similarly to $\sH_{\le c}^\access$. In this case, $(\sO_0)_{\le c}$ identifies with the category denoted $\QLisse(\bP^1)$ in \cite{arinkin2022stacklocalsystemsrestricted}, which is the full subcategory of the unbounded derived category of sheaves consisting of objects whose perverse cohomology groups are (filtered colimits of) lisse sheaves. Infamously, this category is not the derived category of its heart. In fact something more is true, which is that this category is not even generated under colimits and shifts by objects in its heart (see \textit{loc. cit.} \S E.2.6). 

The category $(\sO_0)_{\le c}^\access$ is the category denoted in \textit{loc. cit.} by $\mathrm{IndLisse}(\bP^1)$. It is exactly the full subcategory of $\QLisse(\bP^1)$ generated under colimits and shifts by objects in the heart. Here is an example of a sheaf $\sF$ which lies in $(\sO_0)_{\le c}^\inaccess$: consider an object such that 
\[H^p_i(\sF) = L_s\]
for all $i$, and with the property that the map
\[L_s[i] \to \tau_{\le i}\sF \totext{\p} \tau_{> i}\sF[1] \to L_s[i + 2]\]
is a nonzero element in $H^2(\bP^1)$ for all $i$. Then 
\[\Hom(L_s, \sF) = 0\]
so $\sF$ lies in $(\sO_0)_{\le c}^\inaccess$.

\begin{proposition}\label{prop:convolution-cpt-constr}
Suppose that $\sF$ is a compact object of $\Sh(B\bs G /B)$ and $\sG$ is constructible. Then $\sF * \sG$ and $\sG * \sF$ are both compact. 
\end{proposition}

\begin{proof}
We may assume that $\sF$ is of the form $\pi_! \sF_0$ for some $\sF_0 \in \Sh(G /B)$. By base change for the diagram 
\[
\begin{tikzcd}
G \times^B G/B \arrow[r] \arrow[d] & B \bs G \times^B G / B \arrow[d] \\
G/B \times B \bs G /B \arrow[r] & B \bs G / B \times B \bs G / B
\end{tikzcd}
\]
it suffices to show that the pullback of $\sF_0 \boxtimes \sG$ to $G \times^B G/B$ is compact. Because $G \times^B G/B$ is a scheme, its enough to check this after pulling back along the smooth cover $G \times^B G \to G \times^B G/B$. Then the result follows because the pullback of $\sG$ to $B \bs G$ is compact. 
\end{proof}

\begin{proposition}
$\sH_{\le c}^\access$ is a two-sided monoidal ideal in $\sH_{\le c}$.
\end{proposition}

\begin{proof}
It suffices to check that $x * y$ and $y * x$ lie in $\sH_{\le c}^\access$, where $x$ is a generator of $\sH_{\le c}$ and $y$ is a generator of $\sH_{\le c}^\access$. In particular, we may take $x = \tau_{\le c}\pi_!P_w$ for $w \le c$ and take $y = L_v$ for $v \le c$. Then we note that 
\[x * y = \pi_!P_w * L_v\]
is constructible, and similarly for $y * x$.  
\end{proof}

\subsubsection{} Now define 
\[\sH_c^\access = \sH_c \underset{\sH_{\le c}}{\otimes} \sH_{\le c}^\access.\]
This is equivalently the full subcategory of $\sH_c^\access$ generated by $\tau_c(L_w)$ for $w \in c$.

\section{A categorification of Lusztig's homomorphism}

\subsubsection{} Recall from \S\ref{sss:intro-lusztig-hom} the Lusztig's homomorphism
\[\psi_c: H \to J_c[t^{\pm 1}]\]
such that 
\[\psi_c(b_w) = \sum_{\substack{z \in c \\  d \in \sD_c}} h_{w, d, v}(t)j_v.\]
The fact that this map defines an algebra homomorphism is highly non-obvious, and in fact the following example shows that the map $\psi_c$ cannot be naively categorified to produce a monoidal functor from $\sH$ to $\sJ_c$.

\medskip 

Let $G$ be $SL_3$ and let $s$ and $t$ denote the simple reflections of the Weyl group. Take $c$ to be the two-sided cell $\set{s, t, st, ts}$. Then 
\[\psi_c(b_s) = j_{st} - (t + t^{-1})j_s.\]
However, the object $L_s \in \sH$ is self-dual, while the object 
\[J_{st} \oplus J_s\bracket{-1}[1] \oplus J_s\bracket{1}[-1] \in \sJ_c^\Gm\]
is not self-dual (because $J_{st}^\vee = J_{ts}$), so there cannot be a monoidal functor from $\sH$ to $\sJ_c^\Gm$ sending $L_s$ to $J_{st} \oplus J_s\bracket{-1}[1] \oplus J_s\bracket{1}[-1]$. 

\medskip

Nevertheless, in this section, we will show that after performing an additional categorification it is possible to construct an analogue of Lusztig's homomorphism. Namely, we will show that there is a map 
\[\sH\tmod \to \sJ_c\tmod.\]
In \S\ref{s:6}, we will apply trace of Frobenius to this map to recover the classification of unipotent representations.

\medskip 

The outline of this section is as follows: 
\begin{itemize}
    \item In \S\ref{ss:5-morita} we will construct a monoidal category $\sB_c$ which is Morita-equivalent to $\sH_c^\access$. 
    \item In \S\ref{ss:A1-action-analysis} we show that the action of $\A^1$ on $\sB_c$ is strictly compatible with the monoidal structure. 
    \item In \S\ref{ss:5-rees} we prove some technical results showing that two \textit{a priori} different monoidal structures on $\sJ_c$ are the same. 
\end{itemize}

\subsection{A monoidal category Morita equivalent to $\sH_c^\access$}\label{ss:5-morita}

\newcommand{\Rees}{\mathrm{Rees}}
\newcommand{\enh}{\mathrm{enh}}

\subsubsection{}\label{sss:algebra-structure-gr0} If $\sA$ is a category acted on by $\A^1$ and $\sA$ has a monoidal structure which is left-compatible with the $\A^1$-action, then we have a diagram 
\[
\begin{tikzcd}
& \sA^\Gm_{\le 0} \arrow[dl] \arrow[dr, "F"]& \\
\sA^\Gm & & \sA^{\A^1}
\end{tikzcd}
\]
where both arrows are monoidal by \Cref{constr:monoidal-structure-A1-invts}. The inclusion of $\sA^{\A^1}$ into $\sA^\Gm$ identifies with 
\[\sA^{\A^1} \totext{F^R} \sA^\Gm_{\le 0} \to \sA^\Gm,\]
so it acquires a lax-monoidal structure. As a result, the unit of $\sA^{\A^1}$, which is by definition $\gr_0(1)$, acquires the structure of an algebra object in $\sA^\Gm$. 

\subsubsection{}\label{sss:A1-action-Bc-defn} Applying the paradigm of \S\ref{sss:algebra-structure-gr0} to the case when $\sA = \sH_c$, we obtain an algebra structure on $\gr_0(1_c)$. The category $\sB_c^\Gm = \gr_0(1_c)\bimod(\sH_c^\Gm)$ is monadic over $\sH_c^\Gm$, and the monad preserves the subcategory $\sH_{c, \le 0}^\Gm$. It follows from \Cref{prop:left-lax-equiv-monad} that $\sB_c = \gr_0(1_c)\bimod(\sH_c)$ admits a unique $\A^1$-action such that 
\[\oblv: \sB_c \to \sH_c\]
is $\A^1$-equivariant. Similar considerations apply to the categories of left- and right-modules for $\gr_0(1_c)$ in $\sH_c$. 

\medskip 

The remainder of the current subsection is devoted to proving that $\sB_c$ is Morita equivalent to $\sH_c^\access$. 

\begin{proposition}\label{prop:Bc-semi-rigid}
The category $\sB_c$ is semi-rigid. 
\end{proposition}

\begin{proof}
The functor 
\[\mult: \sB_c \otimes \sB_c \to \sB_c\]
is obtained from 
\begin{equation}\label{eq:multiplication-for-bimod-cat}
(-) \underset{\gr_0(1_c)}{*} (-): \gr_0(1_c)\tmod^r(\sH_c) \otimes \gr_0(1_c)\tmod(\sH_c) \to \sH_c    
\end{equation}
by applying 
\[\gr_0(1_c)\tmod \underset{\sH_c}{\otimes} (-) \underset{\sH_c}{\otimes} \gr_0(1_c)\tmod^r(\sH_c),\]
so it suffices to show that (\ref{eq:multiplication-for-bimod-cat}) admits a continuous $\sH_c$-bilinear right adjoint. 

$\gr_0(1_c)\tmod(\sH_c)$ is compactly generated by objects of the form $\gr_0(1_c) * x$, where $x$ is compact in $\sH_c$. Similarly, $\gr_0(1_c)\tmod^r(\sH_c)$ is compactly generated by objects of the form $y * \gr_0(1_c)$, so to show that (\ref{eq:multiplication-for-bimod-cat}) admits a continuous right adjoint it's enough to show that 
\[y * \gr_0(1_c) * x\]
is compact in $\sH_c$. Since $\sH_c$ is compactly generated by objects of the form $\tau_{\le c}\pi_! P_w$ for $w \in W_c$, it's enough to show that 
\[\tau_{\le c}\pi_! P_w * \gr_0(1_c) * \tau_{\le c}\pi_! P_v\]
is compact. Since $\gr_0(1_c)$ already lies in $\sH_{\le c}$, this product identifies with 
\[\pi_! P_w * \gr_0(1_c) * \pi_! P_v.\]
Now compactness follows from by \Cref{prop:convolution-cpt-constr}, which says that the convolution of a constructible object with a compact object of $\sH$ is again compact. Recall (see \S\ref{ss:defn-Hc}) that $\sH_c$ is semi-rigid, and (\ref{eq:multiplication-for-bimod-cat}) is evidently $\sH_c$-bilinear, so we know that the right adjoint of (\ref{eq:multiplication-for-bimod-cat}) is also $\sH_c$-bilinear. 
\end{proof}

\begin{proposition}\label{prop:semi-rigid-duality-left-right-mod}
Suppose that $\sA$ is a monoidal category and $x \in \Alg(\sA)$ is an algebra in $\sA$. Let $\sB$ denote $x\textrm{-bimod}(\sA)$. Then $x\tmod(\sA)$ is left-dualizable as a $(\sB, \sA)$-bimodule, with left dual $x\tmod^r(\sA)$. 
\end{proposition}

\begin{proof}
The unit of the duality is the natural isomorphism 
\[\epsilon: \sB \simeq x\tmod(\sA) \underset{\sA}{\otimes} x\tmod^r(\sA).\]
The counit 
\[\eta: x\tmod^r(\sA) \underset{\sB}{\otimes} x\tmod(\sA) \to \sA\]
is the map coming from $(-) \underset{x}{*} (-)$. The triangle identities follow by unwinding the definitions. 
\end{proof}

\begin{proposition}\label{prop:Bc-morita-equiv}
The category $\sB_c$ is Morita equivalent to $\sH_c^\access$. 
\end{proposition}

\begin{proof}
\Cref{prop:semi-rigid-duality-left-right-mod} shows that 
\[(\gr_0(1_c)\tmod^r(\sH_c), \gr_0(1_c)\tmod(\sH_c))\] 
is a pair of adjoint functors between $\sH_c\tmod$ and $\sB_c\tmod$. Note that the counit map 
\[\eta: \gr_0(1_c)\tmod^r(\sH_c) \underset{\sB_c}{\otimes} \gr_0(1_c)\tmod(\sH_c) \to \sH_c\]
admits a continuous $(\sH_c, \sH_c)$-bilinear right adjoint: indeed, bilinearity is automatic from the semi-rigidity of $\sH_c$, and because $\sB_c$ is rigid, it suffices to check that the map 
\begin{equation}\label{eq:convolution-map-module-cats}
\gr_0(1_c)\tmod^r(\sH_c) \otimes \gr_0(1_c)\tmod(\sH_c) \to \sH_c
\end{equation}
admits a continuous right adjoint. The category $\gr_0(1_c)\tmod^r(\sH_c)$ is compactly generated by objects of the form $x = x_0 * \gr_0(1_c)$ with $x_0$ compact in $\sH_c$, and similarly for $\gr_0(1_c)\tmod(\sH_c)$. The map (\ref{eq:convolution-map-module-cats}) sends 
\[x \otimes y \mapsto x \underset{\gr_0(1_c)}{*} y,\]
and if $x = x_0 * \gr_0(1_c)$ and $y = \gr_0(1_c) * y_0$, the image is $x_0 * \gr_0(1_c) * y_0$, which is compact in $\sH_c$ by \Cref{prop:convolution-cpt-constr}. 

It follows by \cite[Theorem 1.3]{ben2020highest} (see also \cite[Proposition 4.2.4]{beraldo2024coherentsheavessheareddmodules}) that $\eta$ is fully faithful, so it remains to check that its essential image is $\sH_c^\access$. Because $x_0 * \gr_0(1_c) * y_0$ lies in $\sH_c^\access$, we know that 
\[\gr_0(1_c)\tmod^r(\sH_c) \underset{\sB_c}{\otimes} \gr_0(1_c)\tmod(\sH_c) \subseteq \sH_c^\access\]
Conversely, given any $w \in c$, the irreducible $L_w$ lies in the image of 
\[\oblv: \gr_0(1_c)\tmod(\sH_c) \to \sH_c,\]
so every sheaf in $\sH_c^\access$ lies in the image of $\eta$. 
\end{proof}

\subsection{Analysis of the $\A^1$-action}\label{ss:A1-action-analysis}

\subsubsection{} The goal of this section is to analyze the action of $\A^1$ on $\sB_c$, which was defined in \S\ref{sss:A1-action-Bc-defn}. 

\begin{proposition}\label{prop:Bc-A1-invts-Jc}
\[\sB_c^{\A^1} \simeq \sJ_c.\]
\end{proposition}

\begin{proof}
\Cref{prop:left-lax-equiv-monad} shows that, $\sB_c^{\A^1}$ is monadic over $\sH_c^{\A^1} = \sJ_c$, and the monad identifies with the zeroth graded piece of 
\[\gr_0(1_c) * (-) * \gr_0(1_c).\]
However, this is just the identity endofunctor of $\sJ_c$, so we see that the equivalence
\[\sB_c^{\A^1} \simeq \sJ_c\]
follows from the following (well-known) \Cref{lem:algebra-structure-unit}, applied to the monoidal category $\sC = \End(\sJ_c)$. 
\end{proof}

\begin{lemma}\label{lem:algebra-structure-unit}
Suppose that $\sC$ is a monoidal category. If $x \in \Alg(\sC)$ is an algebra object such that $x$ is isomorphic to the unit $1$ as a plain object of $\sC$, then the unit map $1 \to x$ is an isomorphism in $\Alg(\sC)$.
\end{lemma}

\begin{proof}
It suffices to show that the unit map $1 \to x$ is an isomorphism in $\sC$. If we write $\alpha: x \simeq 1$ for the given isomorphism between $x$ and $1$, then it suffices to show that 
\[x \totext{\alpha} 1 \totext{\unit} x\]
is an isomorphism. Because $\End(x) \simeq \End(1)$ is an $\E_2$-algebra, once we know that endomorphism has a left-inverse, it follows that it is invertible. Tensoring $\unit: 1 \to x$ with $\alpha: x \to 1$, we obtain a commutative square 
\[
\begin{tikzcd}
x \arrow[d, "\alpha"'] \arrow[r, "\unit \otimes \id"] & x \otimes x \arrow[d, "\id \otimes \alpha"] \\
1 \arrow[r, "\unit"'] & x
\end{tikzcd}
\]
so the existence of the desired left-inverse follows from the fact that $\mult$ is a left inverse of $\unit \otimes \id$. 
\end{proof}

\subsubsection{} We now wish to show that the $\A^1$ action on $\sB_c$ is strictly compatible with the monoidal structure on $\sB_c$. First, we note that the argument of \Cref{prop:Bc-A1-invts-Jc} also applies to the categories of left- and right-modules for $\gr_0(1_c)$ in $\sH_c$. This implies the following technical result which will be useful shortly: 

\begin{proposition}\label{prop:deg-zero-summand}
Suppose that $x \in \gr_0(1_c)\tmod^r(\sH_c^\Gm)$ is concentrated in a single graded degree. Then any summand of $x$ as a plain object of $\sH_c^\Gm$ is naturally a summand of $x$ as a right $\gr_0(1_c)$-module.
\end{proposition}

\begin{proof}
Follows from the fact that $\gr_0(1_c)\tmod^r(\sH_c)^{\A^1} \simeq \sJ_c \simeq \sH_c^{\A^1}$, which is proved like in \Cref{prop:Bc-A1-invts-Jc}. 
\end{proof}

\subsubsection{}\label{sss:left-cells} For any $w \in c$, the object $L_w$ is concentrated in a single graded degree, so it naturally acquires the structure of a $\gr_0(1)$-bimodule. Our next result will be a technical statement about the convolution 
\[L_w \underset{\gr_0(1)}{*} L_v.\]
Before we state it, let us recall some facts about left cells:
\begin{enumerate}
    \item Every left cell contains a unique Duflo involution \cite[Theorem 1.10]{LUSZTIG1987536} (see also \cite[Theorem 6.11]{Curtis1988}).
    \item\label{item:sequence} If $w$ and $v$ lie in the same left cell, then there is a sequence of elements 
    \[v = u_0, u_1, \ldots, u_n = w\]
    with each $u_i$ lying in $c$, such that for each $i$ between $1$ and $n$ there is a simple reflection $s_i$ such that $s_i u_{i - 1} > u_{i - 1}$ (for the Bruhat order), and $L_{u_i}$ appears as a direct summand of $L_{s_i} * L_{u_{i - 1}}$ \cite[Lemma 5.3.(ii)]{Curtis1988} (of course this follows from the material of \cite{kazhdan1979representations}).
\end{enumerate}

\begin{proposition}\label{prop:weights-bounded-from-below}
For $w, v \in c$, the weights of 
\[L_w \underset{\gr_0(1)}{*} L_v\]
are bounded from below. In other words, this object lies in $\sH_{c, \ge N}^\Gm$ for some $N$. 
\end{proposition}

\begin{proof}
If $w$ is a Duflo involution, then it is a summand of $\gr_0(1_c)$ (necessarily as a right $\gr_0(1_c)$-module by \Cref{prop:deg-zero-summand}). It follows that 
\[L_w \underset{\gr_0(1)}{*} L_v\]
is a summand of 
\[\gr_0(1)\underset{\gr_0(1)}{*} L_v \simeq L_v\]
so the result follows when $w$ is a Duflo involution. Now suppose that $w$ is an arbitrary element of $c$. Let $d$ denote the unique Duflo involution lying in the left cell containing $w$. We know that there is some sequence 
\[d = u_0, u_1, \ldots, u_n = w\]
as in \S\ref{sss:left-cells}(\ref{item:sequence}). Arguing by induction on the minimal length of such a sequence, we may assume that there is some $x \in c$ and some simple reflection $s$ such that 
\[L_x \underset{\gr_0(1)}{*} L_w\]
has weights bounded from below, $sx > x$, and $L_w$ appears as a direct summand of $L_s * L_x$. By the assumption that $sx > x$, we know that $L_s * L_x$ is concentrated in a single graded degree, so $L_w$ is a summand of $L_s * L_x$ as a right $\gr_0(1)$-module. It follows that 
\[L_w \underset{\gr_0(1)}{*} L_v\]
is a summand of 
\[(L_s * L_x) \underset{\gr_0(1)}{*} L_v \simeq L_s * \left(L_x \underset{\gr_0(1)}{*} L_v\right),\]
so its weights are bounded from below. 
\end{proof}

\begin{lemma}\label{lem:single-degree}
Suppose that $\sA$ is a category acted on by $\A^1$. Suppose that $\sA$ is monoidal compatibly with the action of $\Gm$, and that the monoidal structure is left-lax $\A^1$-equivariant. If the unit of $\sA^{\Gm}$ lies in degree zero and the multiplication preserves $\sA^\Gm_{= 0}$, then the $\A^1$ action on $\sA$ is strictly compatible with the monoidal structure. 
\end{lemma}

\begin{proof}
The only thing that needs to be checked is that the monoidal structure on $\sA^\Gm$ preserves $\sA^\Gm_{\ge 0}$. But every object of $\sA^\Gm_{\ge 0}$ can be written as a colimit of terms concentrated in a single degree, so the result follows from the compatibility of $\mult^\Gm$ with colimits and the action of $\sO(n)$. 
\end{proof}

\begin{proposition}\label{prop:A1-action-strict}
The $\A^1$-action on $\sB_c$ is strictly compatible with the monoidal structure. 
\end{proposition}

\begin{proof}
It is obvious that $\gr_0(1)$ lies in degree zero. By \Cref{lem:single-degree} it remains to show that for any irreducibles $L_w$, $L_v$ the map 
\begin{equation}\label{eq:lax-monoidal-arrow}
L_w \underset{\gr_0(1)}{*} L_v \to \gr_0(L_w * L_v)
\end{equation}
is an isomorphism. It's enough to do this after forgetting the $\gr_0(1_c)$-bimodule structure on both sides. The morphism $1_c \to \gr_0(1_c)$ of algebras in $\sH_c^\Gm$ induces a map of simplicial objects 
\[\mathrm{Bar}(L_w, 1_c, L_v) \to \mathrm{Bar}(L_w, \gr_0(1_c), L_v)\] 
which becomes an isomorphism after applying $\tau_{\ge 0}$. This shows that (\ref{eq:lax-monoidal-arrow}) becomes an isomorphism after applying $\tau_{\ge 0}$, so it remains to show that 
\[\gr_k\left(L_w \underset{\gr_0(1)}{*} L_v\right) = 0\]
for all $k < 0$. Because there are only finitely many elements in $c$, \Cref{prop:weights-bounded-from-below} shows that there is some $N$ such that 
\[\gr_k\left(L_w \underset{\gr_0(1)}{*} L_v\right) = 0\] 
for all $w, v \in c$ and $k < N$. Now it suffices to show that if $k < 0$, and 
\[\gr_\ell\left(L_w \underset{\gr_0(1)}{*} L_v\right) = 0\]
for all $\ell < k$ and all $w \in c$, then 
\[\gr_k\left(L_w \underset{\gr_0(1)}{*} L_v\right) = 0.\]

To see this, consider the object $L_w * \gr_0(1)$, the free right $\gr_0(1)$-module on $L_w$. By Verdier duality we know that $L_w * \gr_0(1)$ lies in weight $\ge -2a$ and $\gr_{-2a}(L_w * \gr_0(1)) \simeq L_w[2a]$. Additionally, 
\[\left(L_w * \gr_0(1)\right) \underset{\gr_0(1)}{*} L_v \simeq L_w * L_v,\]
so we obtain an exact triangle
\[L_w[2a]\bracket{-2a} \underset{\gr_0(1_c)}{*} L_v \to L_w * L_v \to \tau_{> -2a} (L_w * \gr_0(1)) \underset{\gr_0(1)}{*} L_v.\]
Since $L_w * L_v$ lies in degree $\ge -2a$, for any $k < 0$ we obtain an isomorphism 
\[\gr_{k - 2a}\left(\tau_{> -2a} (L_w * \gr_0(1)) \underset{\gr_0(1)}{*} L_v\right)[-1] \simeq \gr_{k - 2a}\left(L_w\bracket{-2a} \underset{\gr_0(1_c)}{*} L_v\right)[2a]\]
so in other words 
\[\gr_k\left(L_w\underset{\gr_0(1_c)}{*} L_v\right) \simeq \gr_{k - 2a}\left(\tau_{> -2a} (L_w * \gr_0(1)) \underset{\gr_0(1)}{*} L_v\right)[-1 - 2a].\]
$\tau_{> -2a} (L_w * \gr_0(1))$ is an iterated extension of objects of the form 
\[L_u[i]\bracket{-i}\]
for $i < 2a$. By our inductive assumption
\[\gr_{k - 2a}\left(L_u[i]\bracket{-i} \underset{\gr_0(1)}{*} L_v\right)[-1 - 2a] \simeq \gr_{k - 2a + i}\left(L_u \underset{\gr_0(1)}{*} L_v\right)[i - 1 - 2a] = 0,\]
so the result follows. 
\end{proof}

\subsection{The monoidal structure of the category of invariants}\label{ss:5-rees}

\subsubsection{} If $\sA$ is a monoidal category acted on by $\A^1$ (strictly), the category $\sA^{\A^1}$ acquires a monoidal structure by functoriality. On the other hand, because the $\A^1$-action is left-compatible with the monoidal structure (in the sense of \S\ref{sss:left-compat}), \Cref{constr:monoidal-structure-A1-invts} equips $\sA^{\A^1}$ with an \textit{a priori} different monoidal structure. The goal of this subsection is to show that these two monoidal structures agree.  

\begin{proposition}\label{prop:localization-two-monoidal-structures}
Suppose that $\sA$ is a monoidal category, equipped with a semiorthogonal decomposition 
\[\sA_\ell \rightleftarrows \sA \rightleftarrows \sA_r\]
such that $\sA_r$ is a (unital) monoidal subcategory of $\sA$ and $\sA_\ell$ is a two-sided ideal. Then the two monoidal structures on $\sA_r$ (as a subcategory and as a quotient of $\sA$) agree.
\end{proposition}

\newcommand{\sub}{\mathrm{sub}}
\newcommand{\quot}{\mathrm{quot}}

\begin{proof}
Let $\sA_r^\sub$ denote $\sA_r$, viewed as having the subcategory monoidal structure, and let $\sA_r^\quot$ denote $\sA_r$ with the quotient monoidal structure. If we write $i: \sA_r \to \sA$ for the inclusion, then the composite 
\[\sA_r^\sub \totext{i} \sA \totext{i^L} \sA_r^\quot\]
is a monoidal equivalence between $\sA_r^\sub$ and $\sA_r^\quot$.
\end{proof}

\begin{proposition}\label{prop:monoidal-structure-A1-invts}
Suppose that $\sA$ is a monoidal category acted on by $\A^1$ (strictly). Then the monoidal structure on $\sA^{\A^1}$ of \Cref{constr:monoidal-structure-A1-invts} realizes $\sA^{\A^1}$ as the $\A^1$-invariants of the monoidal category $\sA$. 
\end{proposition}

\begin{proof}
Regarding $\sA^{\A^1}$ as the $\A^1$-invariants of the monoidal category $\sA$, there is a canonical monoidal functor $i: \sA^{\A^1} \to \sA$. Equivariantizing, we obtain a monoidal functor 
\[\sA^{\A^1} \otimes \Rep(\Gm) \simeq (\sA^{\A^1})^\Gm \totext{i^\Gm} \sA^\Gm\]
(the first equivalence holds because the action of $\A^1$ (and hence $\Gm$) on $\sA^{\A^1}$ is canonically trivial). The composite 
\[\sA^{\A^1} \totext{\mathrm{wt}_0} \sA^{\A^1} \otimes \Rep(\Gm) \totext{\oblv} \sA^{\A^1}\]
is a monoidal equivalence (as it is a composite of monoidal functors, and it is an equivalence). Moreover, the functor 
\[\sA^{\A^1} \totext{\mathrm{wt}_0} \sA^{\A^1} \otimes \Rep(\Gm) \totext{i^\Gm} \sA^\Gm \]
is fully faithful, so we see that the monoidal structure on $\sA^{\A^1}$ coming from the fact that it is the $\A^1$-invariants of $\sA$ identifies with the monoidal structure it acquires as a subcategory of $\sA^\Gm$. Now the result follows by applying \Cref{prop:localization-two-monoidal-structures} to 
\[\sA^\Gm_{\le -1} \rightleftarrows \sA^\Gm_{\le 0} \rightleftarrows \sA^{\A^1}.\]
\end{proof}

\begin{corollary}\label{prop:gr-monoidal}
The functor $\gr_\sbullet: \sA^\Gm \to \sA^{\A^1} \otimes \Rep(\Gm)$ acquires a natural monoidal structure making the following diagrams commute 
\begin{equation}\label{eq:A1-triangle-1}
\begin{tikzcd}
& \sA^{\A^1} \otimes \Rep(\Gm) \arrow[rd, "\oblv"]& \\
\sA^\Gm \arrow[rr, "(\act_0)^\Gm"'] \arrow[ru, "\gr_\sbullet"] & & \sA^\Gm
\end{tikzcd}    
\end{equation}
\begin{equation}\label{eq:A1-triangle-2}
\begin{tikzcd}
\sA^{\A^1} \otimes \Rep(\Gm) \arrow[dr, "\oblv"'] \arrow[rr, "\id"] & & \sA^{\A^1} \otimes \Rep(\Gm) \\
& \sA^\Gm \arrow[ur, "\gr_\sbullet"']&
\end{tikzcd}    
\end{equation}
(as diagrams of monoidal functors). 
\end{corollary}

\newcommand{\deeq}{\mathrm{deeq}}

\begin{proof}
By \Cref{prop:invts-idempotent}, there is a monoidal functor $r: \sA \to \sA^{\A^1}$, making the diagrams 
\[
\begin{tikzcd}
& \sA^{\A^1} \arrow[rd, "i"]& \\
\sA\arrow[rr, "\act_0"'] \arrow[ru, "r"] & & \sA
\end{tikzcd}
\quad
\begin{tikzcd}
\sA^{\A^1}  \arrow[dr, "i"'] \arrow[rr, "\id"] & & \sA^{\A^1}  \\
& \sA \arrow[ur, "r"']&
\end{tikzcd}
\]
commute (as $\A^1$-equivariant monoidal categories). Forgetting the monoidal structure, we obtain similar diagrams of $\A^1$-equivariant categories. However, the de-equivariantizations of (\ref{eq:A1-triangle-1}, \ref{eq:A1-triangle-2}) are diagrams of the same shape, so we obtain natural isomorphisms $r \simeq (\gr_\sbullet)^\deeq$ and $i \simeq \oblv^\deeq$ of $\A^1$-equivariant functors. This gives the desired upgrade of $(\gr_\sbullet)^\deeq$ to an $\A^1$-equivariant monoidal functor. 
\end{proof}

\subsubsection{} We will end by using the monoidal structure on $\gr_\sbullet$ to give a new proof of the rigidity of $\sJ_c$. This was first proved for affine and finite Weyl groups in \cite[\S4.3]{bezrukavnikov2009tensor} by explicitly describing $\sJ_c$ (see \cite[\S5]{EliasWilliamson2021} for more discussion and a proof for more general Coxeter groups). This result will be used later in \S\ref{s:6}. 

\begin{proposition}\label{prop:Jc-rigid}
The functor $\gr_\sbullet$ admits a continuous right adjoint and the category $\sJ_c$ is rigid. 
\end{proposition}

\begin{proof}
$\sB_c^\Gm$ is compactly generated by (shifts and Tate twists of) objects of the form 
\[\gr_0(1_c) * \tau_{\le c}\pi_!P_w * \gr_0(1_c) \simeq \gr_0(1_c) * \pi_!P_w * \gr_0(1_c)\] 
for $w \in c$ (here we have used the fact that $\gr_0(1_c)$, being a direct sum of irreducibles in $c$, already lies in $\sH_{\le c}$). By \Cref{prop:convolution-cpt-constr}, the object $\gr_0(1_c) * \pi_!P_w * \gr_0(1_c)$ is compact in $\sH$, so it lies in finitely many graded degrees and its associated graded pieces are direct sums of finitely many irreducible objects. It follows that $\gr_\sbullet(\gr_0(1_c) * \pi_!P_w * \gr_0(1_c))$ is compact in $\sJ_c \otimes \Rep(\Gm)$. This shows that $\gr_\sbullet$ admits a continuous right adjoint. 

To see that $\sJ_c$ is rigid, first note that the unit $\gr_0(1_c)$ is compact so it suffices to show that every compact object is left- and right-dualizable. Since $\oblv: \sJ_c^\Gm \to \sJ_c$ is monoidal, it suffices to show that $J_w$, viewed as an object of $\sJ_c^\Gm$ living in weight zero, is left- and right-dualizable. However, $J_w$ is a summand of $\gr_\sbullet(\gr_0(1_c) * \pi_!P_w * \gr_0(1_c))$ (being the zeroth graded peice). Since $\gr_0(1_c) * \pi_!P_w * \gr_0(1_c)$ is compact, it is left- and right-dualizable (as $\sB_c^\Gm$ is semi-rigid). By \Cref{prop:gr-monoidal}, applied with $\sA = \sB_c$, we know that $\gr_\sbullet$ is monoidal, which shows that $J_w$ is a summand of a left- and right-dualizable object. 
\end{proof}

\section{Unipotent representations}\label{s:6}

\subsubsection{} The goal of this section is to complete the proof of \Cref{thm:intro-main}. 

\medskip 

The outline of this section is as follows:
\begin{itemize}
    \item In \S\ref{ss:motives} we show that the prior work of Soergel and Wendt on mixed Tate motives implies that the action of Frobenius on $\sH$ identifies with the action of $\sqrt{q}$, with respect to the $\A^1$ action of \S\ref{ss:action-of-A1}.
    \item In \S\ref{ss:6-tr-inaccess} we show that $HH(\sH_c; \Frob)$ agrees with $HH(\sH_c^\access; \Frob)$ using the equivalence $HH(\sH; \Frob) \simeq \Rep(G(\F_q))^\unip$. 
    \item In \S\ref{ss:6-action-semisimple} we show that every action of $\Gm$ on a semisimple category is trivial and give a criterion for an $\A^1$ action to be trivial. 
    \item In \S\ref{ss:Bc-ren} we employ a certain renormalization of $\sB_c$ to show that the action of $\A^1$ on $HH(\sB_c; \Frob)$ is trivial. 
    \item In \S\ref{ss:6-main} we finish the proof of \Cref{thm:intro-main}.
\end{itemize}

\subsection{Motives and the action of $\W$}\label{ss:motives}

\subsubsection{} Recall the group $\W$ from \S\ref{ss:defn-W}. Let $\sH^\naive$ denote the monoidal category which is equal to $\sH$ as a plain category, but whose monoidal structure is given by !-pull (as opposed to perverse pullback) and *-push along the correspondence 
\[
\begin{tikzcd}
& B\bs G \times^B G /B \arrow[dl] \arrow[dr] & \\
B \bs G /B \times B \bs G / B & & B \bs G / B  
\end{tikzcd}
\] 
As in \S\ref{ss:action-of-A1}, $\W$ acts on $\sH^\naive$ compatibly with its monoidal structure. 

\medskip 

Now let $\pi: \W \to \Gm$ denote the map which is Cartier dual to the map $\Z \to \Weil_q$ sending 1 to $q$.

\begin{proposition}
The action of $\W$ on $\sH^\naive$ factors through $\pi: \W \to \Gm$ compatibly with the monoidal structure on $\sH^\naive$. 
\end{proposition}

\begin{proof}
It suffices to show an equivalence of $\Rep(\W)$-linear monoidal categories
\[\sH^{\naive, \W} \simeq \Rep(\W) \underset{\Rep(\Gm)}{\otimes} \sC\]
for some monoidal category $\sC$ acted on by $\Rep(\Gm)$. As in \cite{soergel2018perverse}, let $\sC$ denote the category of $\ol{\Q}_\ell$-linear stratified mixed Tate motives on $B \bs G / B$. This category is monoidal by the compatibility of stratified mixed Tate motives with pushforward along Whitney-Tate maps (see \cite[\S3.1]{richarz2020intersection}). There is a natural $\ell$-adic realization functor $\sC \to \Shv(B\bs G /B)^{\Weil, \mathrm{wt}}$ which is compatible with six functors. In particular, it is monoidal and linear for the action of $DTM(\Spec \F_q)$. By the vanishing of the rational higher $K$-theory of $\F_q$, we know that $DTM(\Spec \F_q) \simeq \Rep(\Gm)$. We therefore obtain a map of $\Rep(\W)$-linear monoidal categories
\[\Rep(\W) \underset{\Rep(\Gm)}{\otimes} \sC \to \sH^{\naive, \W}.\]
It suffices to show that this map is an equivalence after de-equivariantiztion for $\W$, where the result follows from \cite[Theorem 3(2)]{soergel2018perverse}. 
\end{proof}

\subsubsection{Remark} The appearance of motives in the above proof is not really necessary for our present purposes; we will briefly sketch an alternative approach just involving $\ell$-adic sheaves. One could take $\sC$ to be the category of stratified Tate $\ell$-adic sheaves in the sense of \cite[\S7.2]{achar2013koszul} (the category which is called $D^\mathrm{mix}_{\mathscr{S}}(X)$ in \textit{loc. cit.}). Perhaps up to some difficulties with homotopy coherence, the ``genuineness'' results of \cite[\S9]{achar2013koszul} imply that this category is monoidal (and one can imagine getting around the homotopy coherence problem using the t-structure on $D^\mathrm{mix}_{\mathscr{S}}(X)$). Then \cite[Proposition 7.5(ii)]{achar2013koszul} implies that the natural functor $\Rep(\W) \underset{\Rep(\Gm)}{\otimes} \sC \to \sH^{\naive, \W}$ is an isomorphism. 

\subsubsection{}\label{sss:new-monoidal-structure} Suppose that $\sC$ is a monoidal category and $x \in \Alg(\sC)$. We will say that a map 
\[\ell: 1 \to x\]
is an invertible point of $x$ if the composition
\begin{equation}\label{eq:invertible-mult}
x \totext{\id \otimes \ell} x \otimes x \totext{\mult} x    
\end{equation}
is an isomorphism. Given such an $\ell$, the object $x$ acquires a new algebra structure by transport of structure along the isomorphism (\ref{eq:invertible-mult}). Concretely, if $\sC = \dgCat$, $x = \sA$ is a monoidal category, and $\ell = \sL$ is an invertible object of $\sA$, then the new monoidal structure on $\sA$ has multiplication map 
\[x *^\sL y \simeq x * \sL * y\] 
and unit $\sL^{\inv}$. Note that the modified multiplication is indeed associative, because 
\[(x * \sL * y) * \sL * z \simeq x * \sL * (y * \sL * z).\]

\subsubsection{}\label{ss:frob-act-v} If we pick a choice $v$ of a square-root of $q$, we obtain a retract $r$ of the map $\Gm \to \W$. The kernel of the retract lies in the kernel of $\pi$, and the following diagram commutes 
\[
\begin{tikzcd}
\W \arrow[rr, "r"] \arrow[dr, "\pi"'] & & \Gm \arrow[dl, "(-)^2"] \\
& \Gm
\end{tikzcd}
\]
For clarity of notation, going forward we will denote this new copy of $\Gm$ which is the image of $r$ by $\Gm^{1/2}$. 

\begin{proposition}\label{prop:frob-action-Gm}
The action of $\W$ on $\sH$ factors through $r: \W \to \Gm^{1/2}$ compatibly with the monoidal structure on $\sH$. 
\end{proposition}

\begin{proof}
Since the action of $\W$ on $\sH^\naive$ factors through $\Gm$, it factors through $\Gm^{1/2}$ as well. Now apply the construction of \S\ref{sss:new-monoidal-structure} where $\sC = \Gm^{1/2}\textrm{-Cat}$ and $x = \sH^\naive$, and $\ell$ is the object 
\[\sO(\dim B)[-\dim B] \otimes 1_{\sH^\naive}\]
of 
\[\Rep(\Gm^{1/2}) \underset{\Rep(\Gm)}{\otimes} \sH^{\naive} \simeq \sH^{\naive, \Gm^{1/2}} \simeq \Hom_{\Gm^{1/2}}(\Vect, \sH^{\naive}).\]
\end{proof}

\subsection{Calculating the trace of $\sH^\inaccess$}\label{ss:6-tr-inaccess}

\subsubsection{}\label{sss:trace-of-hecke} The arguments of \cite[\S6.3]{eteve2024free} show that 
\[\tr(\Frob; \sH\tmod) \simeq \Rep(G(\F_q))^\unip\]
in such a way that the canonical trace map 
\[\sH \to \tr(\Frob; \sH\tmod)\]
identifies with pull-push along the diagram 
\begin{equation}\label{eq:DL-ind-diagram}
\begin{tikzcd}
& G/\Ad_{\Frob}(B) \arrow[dl] \arrow[dr] & \\
B \bs G/B & & G/\Ad_{\Frob}(G) 
\end{tikzcd}  
\end{equation}
However, this result is not stated as-is in \textit{loc. cit}, so for the sake of completeness we note that this follows directly from \cite[Proposition 8.57]{zhu2025tamecategoricallocallanglands}. Namely, in the notation of \textit{loc. cit}, let $X = BB$, let $Y = BG$, and let $f: X \to Y$ denote the natural map. Let $\phi_X$ and $\phi_Y$ denote the respective Frobenius endomorphisms. Let $Z = Y$, let $g_1 = \id$, and let $g_2 = \phi_Y$. Since $f$ is proper and $\Delta_X$ is smooth, the assumptions of the proposition are satisfied. Therefore we learn that the category
\[HH(\sH; \Frob) \simeq HH(\Sh(X \times_Y X); \Frob)\]
identifies with the full subcategory of 
\[\Sh(Y \times_{Y \times Y} Z) \simeq \Sh(G / \Ad_{\Frob}(G)) \simeq \Rep(G(\F_q))\]
generated by pull-push along 
\[
\begin{tikzcd}
& X  \times_{Y \times Y} Z \arrow[dl] \arrow[dr] & \\
X \times X \times_{Y \times Y} Z & & Y \times_{Y \times Y} Z
\end{tikzcd}
\]
in such a way that the canonical trace functor identifies with pull-push along this diagram. However, after unwinding definitions this push-pull diagram identifies with (\ref{eq:DL-ind-diagram}). The functor of pull-push along (\ref{eq:DL-ind-diagram}) is precisely Deligne-Lusztig induction with trivial central character, so indeed
\[HH(\sH; \Frob) \simeq \Rep(G(\F_q))^\unip\]
(essentially by the definition \cite[Definition 7.8]{deligne1976representations} of what it means for a representation to be unipotent).

\begin{proposition}\label{prop:tr-inaccess}
\[\tr(\Frob; \sH_{\le c}^\inaccess\tmod) = 0.\]
\end{proposition}

\begin{proof}
Applying \Cref{constr:semiorthogonal-trace} to the semiorthogonal decomposition 
\begin{equation}\label{eq:access-semiorthogonal}
\sH_{\le c}^\access \rightleftarrows \sH_{\le c} \rightleftarrows \sH_{\le c}^\inaccess    
\end{equation}
we see that it's enough to show that the composite
\[\tr(\Frob; \sH_{\le c} \tmod) \totext{\oblv} \sH_{\le c} \to \sH_{\le c}^\access\]
is conservative, where the second arrow is the natural projection from (\ref{eq:access-semiorthogonal}). Because $\tr(\Frob; \sH_{\le c} \tmod)$ is semisimple, it's enough to show that the image of any simple object under $\oblv$ is constructible. We can check this after including $\sH_{\le c}$ into $\sH$, and applying \Cref{constr:semiorthogonal-trace} to the semiorthogonal decomposition
\[\sH_{\not\le c} \rightleftarrows \sH \rightleftarrows \sH_{\le c}\]
we see that 
\[\tr(\Frob; \sH_{\le c} \tmod) \totext{\oblv} \sH_{\le c} \to \sH\]
identifies with 
\[\tr(\Frob; \sH_{\le c} \tmod) \to \tr(\Frob; \sH\tmod) \totext{\oblv} \sH.\]
\S\ref{sss:trace-of-hecke} shows that $\oblv$ identifies with pull-push along the diagram 
\[
\begin{tikzcd}
& G/\Ad_\Fr(B) \arrow[dl] \arrow[dr] & \\
G/\Ad_\Fr(G) & & B \bs G/B
\end{tikzcd}
\]
and because $G/\Ad_\Fr(B) \to B \bs G/B$ is schematic, *-push preserves constructibility. 
\end{proof}

\begin{proposition}\label{prop:hh-hc-equals-hh-bc}
\[HH(\sH_c; \Frob) \simeq HH(\sB_c; \Frob)\]
\end{proposition}

\begin{proof}
Since 
\[\sH_c^\inaccess = \sH_c \underset{\sH_{\le c}}{\otimes} \sH_{\le c}^\inaccess\]
is a coreflective subcategory of $\sH_{\le c}^\inaccess$, it follows from \Cref{constr:semiorthogonal-trace} and \Cref{prop:tr-inaccess} that 
\[HH(\sH_c^\inaccess; \Frob) = 0.\]
Applying \Cref{constr:semiorthogonal-trace} again we see that 
\[HH(\sH_c; \Frob) \simeq HH(\sH_c^\access; \Frob).\]
Now the result follows from \Cref{prop:Bc-morita-equiv}
\end{proof}

\subsection{Actions on semisimple categories}\label{ss:6-action-semisimple}

\begin{proposition}\label{prop:gm-action-trivial}
Suppose that $\sC$ is a split semisimple dg category. Every action of $\Gm$ on $\sC$ is (noncanonically) trivializable.
\end{proposition}

\begin{proof}
Let's write $\sC \simeq \Vect^I$ and let $\ins_i: \Vect \to \sC$ and $\ev_i: \sC \to \Vect$ denote insertion and evaluation functors corresponding to $i \in I$. Let $\sF^i_j \in \QCoh(\Gm)$ denote 
\[\ev_j \circ\ \coact \circ \ins_i(k).\]
The definition of a coaction implies that 
\begin{align*}
i_1^*(\sF^i_j) &\simeq k^{\delta_{ij}} \\
\mult^*(\sF^i_j) &\simeq \bigoplus_k \sF^i_k \boxtimes \sF^k_j
\end{align*}
By the commutativity of the following diagram 
\[
\begin{tikzcd}
\Gm \arrow[d] \arrow[r, "\id \times \mathrm{inv}"] & \Gm^2 \arrow[d, "\mult"] \\
* \arrow[r, "i_1"'] & \Gm
\end{tikzcd}
\]
we learn that 
\begin{equation}\label{eq:gm-action}
\bigoplus_k \sF^i_k \otimes \mathrm{inv}^*(\sF^k_j) \simeq (\sO_{\Gm})^{\delta_{ij}}.  
\end{equation}
Because $\sO_\Gm$ is indecomposable as an object of $\QCoh(\Gm)$, setting $j = i$ in \Cref{eq:gm-action} we learn that for each $i$ there is a unique $k$ such that 
\begin{equation}\label{eq:gm-action-3}
\sF^i_k \otimes \mathrm{inv}^*(\sF^k_i) \simeq \sO_\Gm    
\end{equation}
and for all other values of $k$, the sheaf $\sF^i_k$ is zero. Note in particular that \Cref{eq:gm-action-3} implies that the sheaf $\sF^i_k$ is invertible. Setting $j = k$ in \Cref{eq:gm-action}, we see that 
\begin{equation}\label{eq:gm-action-2}
\sF^i_k \otimes \mathrm{inv}^*(\sF^k_k) \oplus \sG \simeq (\sO_\Gm)^{\delta_{ik}}    
\end{equation}
for some sheaf $\sG$. Since $i_1^*(\sF^k_k) \neq 0$, it follows that the left hand side of \Cref{eq:gm-action-2} is nonzero, which implies that $k = i$. All in all, we have shown that 
\[
\sF^i_j = \begin{cases}
\sL_i & i = j \\
0 & i \neq j
\end{cases}
\]
where $\sL_i$ is some invertible sheaf in $\QCoh(\Gm)$. In fact, the coaction of $\Gm$ on $\sC$ endows each $\sL_i$ with the structure of a multiplicative line bundle over $\Gm$. A trivialization of the action of $\Gm$ on $\sC$ amounts to a trivialization of each $\sL_i$ as a multiplicative line bundle, so the result follows from the well-known fact that every multiplicative line bundle on $\Gm$ is (noncanonically) trivializable.
\end{proof}

\begin{proposition}\label{prop:A1-action-triv}
Suppose that $\sC$ is a split semisimple dg category with an action of $\A^1$. If $\act_0$ is conservative, then the action of $\A^1$ is trivializable. 
\end{proposition}

\begin{proof}
By \Cref{prop:gm-action-trivial}, we know that the action of $\Gm \subseteq \A^1$ is trivializable. It follows that 
\[\sC^\Gm \simeq \Rep(\Gm)^I.\]
Moreover, using the trivialization of the $\Gm$ action on $\sC$, for every $i \in I$ we can endow $k_i = \ins_i(k)$ with a $\Gm$-equivariant structure. Let $n_i$ denote the infimum of the $n$ such that $k_i \in \Rep(\Gm)_{\le n}$. If $n_i = -\infty$, then $\act_0(k_i) = 0$. Otherwise, by changing the trivialization of the $\Gm$ action we may set $n_i = 0$. 
\end{proof}

\subsubsection{} For completeness, let's give an example of a split semisimple category with an action of $\A^1$ such that $\act_0$ isn't conservative. We will take the category $\sC$ to be $\Vect$; it remains to specify the $\A^1$-action. Note that the diagram 
\[
\begin{tikzcd}
\Gm \times \Gm \arrow[d, "\mult"'] \arrow[r, "j \times j"] & \A^1 \times \A^1 \arrow[d, "\mult"] \\
\Gm \arrow[r, "j"'] & \A^1
\end{tikzcd}
\]
is \textit{Cartesian} (here $j: \Gm \to \A^1$ is the canonical embedding). It follows that the functor $j^*: \QCoh(\A^1) \to \QCoh(\Gm)$ is symmetric monoidal (where the source and target have the convolution symmetric monoidal structures), and our strange $\A^1$ action on $\Vect$ is given by restricting the trivial $\Gm$-action along $j^*$. Since $j^*(\delta_0) = 0$, this action evidently has the property that $\act_0 = 0$. 

\subsection{A renormalization of $\sB_c$}\label{ss:Bc-ren}

\newcommand{\constr}{\mathrm{constr}}
\renewcommand{\pr}{\mathrm{pr}}
\newcommand{\unren}{\textrm{un-ren}}
\newcommand{\dual}{\mathrm{dual}}

\subsubsection{} The goal of this section is to prove that $HH(\sB_c; \Frob) \simeq HH(\sJ_c; \id)$. To do this, we first need to introduce an auxiliary category, which we will call $\sB_c^\ren$. 

\medskip

Let $\sB_c^{\Gm, \dual}$ denote the small subcategory of $\sB_c$ consisting of objects which are both left- and right-dualizable. 

\begin{proposition}\label{prop:dual-obj-Bc}
An object $x \in \sB_c^\Gm$ lies in $\sB_c^{\Gm, \dual}$ if and only if $\gr_\sbullet(x)$ is compact in $\sJ_c \otimes \Rep(\Gm)$. 
\end{proposition}

\begin{proof}
Since $\gr_\sbullet$ is monoidal, if $x$ is both left- and right-dualizable then $\gr_\sbullet(x)$ has this property as well, which implies that $\gr_\sbullet(x)$ is compact in $\sJ_c \otimes \Rep(\Gm)$ by \Cref{prop:Jc-rigid}. Conversly, suppose that $\gr_\sbullet(x)$ is compact. It follows that $\gr_\sbullet(x)$ is concentrated in finitely many degrees, and we claim that this implies that $\tau_{\le n}x = 0$ for $n \ll 0$. Indeed, for $n \ll 0$, the object $\tau_{\le n}x$ lies in the intersection 
\[\bigcap_{n \in \Z} \sB_{c, \le n}^\Gm.\]
The composite functor 
\[\sB_c \to \sH_c \to \sH_{\le c} \to \sH\]
is conservative and left-lax $\A^1$-equivariant, so it follows that $\tau_{\le n}x = 0$ for $n$ sufficiently negative. Now it follows that $x$ can be written as a finite colimit of its graded pieces, each of which is dualizable, so $x$ itself must be dualizable.  
\end{proof}

\subsubsection{} Now let $\sB_c^{\Gm, \ren}$ denote $\Ind(\sB_c^{\Gm, \dual})$. Because $\sB_c^\Gm$ is semi-rigid, every compact object fo $\sB_c^\Gm$ lies in $\sB_c^{\Gm, \dual}$. It follows that we obtain a pair of adjoint functors 
\[\ren: \sB_c^\Gm \rightleftarrows \sB_c^{\Gm, \ren}: \unren\]
which exhibit $\sB_c^\Gm$ as a colocalization of $\sB_c^{\Gm, \ren}$. Because $\sB_c^{\Gm, \dual}$ is a small monoidal subcategory of $\sB_c^\Gm$, the category $\sB_c^{\Gm, \ren}$ acquires a monoidal structure such that $\unren$ is monoidal. With respect to this monoidal structure, $\sB_c^{\Gm, \ren}$ is tautologically rigid. 

\medskip 

Because $\sB_c^{\Gm, \dual}$ is stable under the action of $\Perf(B\Gm) \subseteq \Rep(\Gm)$, the category $\sB_c^{\Gm, \ren}$ acquires a $\Rep(\Gm)$-module structure such that $\unren$ is $\Rep(\Gm)$-linear. Let's write $\sB_c^\ren$ for the de-equivariantization of $\sB_c^{\Gm, \ren}$. Finally, if we define 
\[\sB_{c, \le n}^{\Gm, \ren} = \Ind(\sB_c^{\Gm, \dual} \cap \sB_{c, \le n}^\Gm)\]
then this equips $\sB_c^\ren$ with an $\A^1$-action. It follows from \Cref{prop:dual-obj-Bc} that the truncation functors of $\sB_c^\Gm$ preserve $\sB_c^{\Gm, \dual}$, so with respect to this $\A^1$-action the (de-equivariantization of) $\unren$ is strictly $\A^1$-equivariant. Moreover, the monoidal structure of $\sB_c^{\ren}$ is strictly compatible with the action of $\A^1$. 

\subsubsection{} Let's give an example of the difference between $\sB_c$ and $\sB_c^\ren$: when $G = SL_2$ and $c = \set{1}$, we saw in \S\ref{sss:sl2-1-qcoh} that 
\[\sH_c \simeq \QCoh(\Omega \A^1),\]
and this equivalence exchanges the monoidal structure on $\sH_c$ with the convolution monoidal structure on $\Omega \A^1$. Furthermore, for this cell $1_c = \gr_0(1_c)$, so $\sB_c = \sH_c$. Note that the monoidal unit fails to be compact, even though it is of course dualizable (this is perhaps not so surprising, because the monoidal unit of $\sH$ failed to be compact to begin with). On the other hand, 
\[\sB_c^\ren \simeq \ICoh(\Omega \A^1),\]
again with the convolution monoidal structure. With respect to these identifications, the canonical monoidal functor $\sB_c^\ren \to \sB_c$ identifies with 
\[\Psi: \ICoh(\Omega \A^1) \to \QCoh(\Omega \A^1).\]

\begin{proposition}\label{prop:coact-Bren-str-cts}
$\coact: \sB_c^\ren \to \QCoh(\A^1) \otimes \sB_c^\ren$ admits a continuous right adjoint.
\end{proposition}

\begin{proof}
We can check this after taking $\Gm$-invariants. The functor 
\[\coact^\Gm: \sB_c^{\Gm, \ren} \to \QCoh(\A^1/\Gm) \otimes \sB_c^\ren\]
is the composite 
\begin{equation}\label{eq:composite-arrow}
\sB_c^{\Gm, \ren} \to \QCoh(\A^1/\Gm) \otimes \sB_c^{\Gm, \ren} \totext{\id \otimes \oblv} \QCoh(\A^1/\Gm) \otimes \sB_c^\ren    
\end{equation}
where the first functor send $x \in \sB_c^{\Gm, \ren}$ to 
\[(\tau_{\le n}x) \in \sB_c^{\Gm, \ren, \Fil} \simeq \QCoh(\A^1/\Gm) \otimes \sB_c^{\Gm, \ren}.\] 
The second arrow in (\ref{eq:composite-arrow}) has a continuous right adjoint by the rigidity of $\Rep(\Gm)$, so it suffices to show that for any $x \in \sB_c^{\Gm, \dual}$, the filtered object $(\tau_{\le n}x)$ is compact as such. Since $\tau_{\le n}x = 0$ for $n \ll 0$, its enough to show that $\gr_\sbullet(x)$ is compact in $\Rep(\Gm) \otimes \sB_c^{\ren, \Gm}$, and this follows from \Cref{prop:dual-obj-Bc}. 
\end{proof}

\subsubsection{} By \Cref{prop:frob-action-Gm}, the endomorphism $\Frob$ of $\sB_c$ identifies with $\act_{\sqrt{q}}$. In particular, $\Frob$ is $\A^1$-equivariant. By \S\ref{sss:hh-A1-equiv}, it follows that $HH(\sB_c; \Frob)$ acquires a natural $\A^1$-action. 

\begin{proposition}\label{prop:str-cts-semisimple-conservative}
Suppose that $\sC$ and $\sD$ are categories acted on by $\A^1$. Suppose that there is an $\A^1$-equivariant functor $\pi: \sD \to \sC$ which admits a fully-faithful (but not necessarily $\A^1$-equivariant) left adjoint. If $\sC$ is semisimple and the coaction functor 
\[\coact: \sD \to \sD \otimes \QCoh(\A^1)\]
is strongly continuous, then the endofunctor $\act_0$ of $\sC$ is conservative. 
\end{proposition}

\begin{proof}
Slightly abusing notation, we will write $j^*$ and $j_*$ for the restriction functor 
\[\sC \otimes \QCoh(\A^1) \to \sC \otimes \QCoh(\Gm)\]
and its right adjoint, respectively. Similarly, we will write $\hat{\iota}^*$ and $\hat{\iota}_+$ for the restriction  
\[\sC \otimes \QCoh(\A^1) \to \sC \otimes \QCoh(\A^1)_0\]
and its left adjoint. We will also use this notation for the corresponding functors into and out of $\sD \otimes \QCoh(\A^1)$, since it will be clear from context which categories we're mapping out of.  

To show that $\act_0$ is a conservative endomorphism of $\sC$, it suffices to show that for every simple object $x \in \sC$, $\act_0(x) \neq 0$. Note that $i_0^*: \QCoh(\A^1)_0 \to \Vect$ admits a left adjoint whose image generates $\QCoh(\A^1)_0$. It follows that 
\[\sC \otimes \QCoh(\A^1)_0 \totext{\id \otimes i_0^*} \sC\]
is conservative, so if $\act_0(x) = 0$ then $\hat{\iota}_+\hat{\iota}^* \coact_\sC(x) = 0$. 

We have an exact triangle 
\[\hat{\iota}_+\hat{\iota}^* \coact_\sD \pi^L(x) \to \coact_\sD \pi^L(x) \to j_*j^* \coact_\sD \pi^L(x)\]
in $\sD \otimes \QCoh(\A^1)$, and by assumption the middle term is compact ($x$ is compact in $\sC$ because it is simple; $\pi^L(x)$ is compact in $\sD$ because $\pi^L$ has a continuous right adjoint). Because $\pi$ is $\A^1$-equivariant, we know that $\pi^L$ is left-lax $\A^1$-equivariant, and in particular $\Gm$ equivariant. By \Cref{prop:gm-action-trivial}, it follows that 
\[j^*\circ \coact\circ  \pi^L(x) \simeq \pi^L(x) \otimes \sO_\Gm.\] 
If $\hat{\iota}_+\hat{\iota}^* \coact_\sC(x) = 0$ then 
\[\hat{\iota}_+\hat{\iota}^* \coact_\sD(x) \in \sC^\perp\]
which would imply the vanishing of the boundary map
\[j_*j^* \coact_\sD \pi^L(x) \to \hat{\iota}_+\hat{\iota}^* \coact_\sD \pi^L(x)[1].\]
In particular, we would have 
\[\coact_\sD \pi^L(x) \simeq \hat{\iota}_+\hat{\iota}^* \coact_\sD \pi^L(x)  \oplus \left(\pi^L(x) \otimes j_*(\sO_\Gm)\right).\]
This contradicts the compactness of $\coact_\sD \pi^L(x)$, because 
\[\pi^L(x) \otimes j_*(\sO_\Gm) \simeq \colim \pi^L(x) \otimes \sO_{\A^1}\]
but the projection 
\[\coact_\sD \pi^L(x) \to \pi^L(x) \otimes j_*(\sO_\Gm)\]
doesn't factor through any finite term of the colimit. 
\end{proof}

\begin{proposition}\label{prop:A1-hh-Bc-triv}
The action of $\A^1$ on $HH(\sB_c; \Frob)$ is trivializable.
\end{proposition}

\begin{proof}
Since 
\[HH(\sB_c; \Frob) \simeq \Rep(G(\F_q))^\unip_c\]
is semisimple, by \Cref{prop:A1-action-triv} it suffices to show that the endofunctor $\act_0$ of $HH(\sB_c; \Frob)$ is conservative. By \Cref{constr:semiorthogonal-trace} and \Cref{prop:hh-map}, we have a diagram 
\[HH(\unren; \Frob)^L: HH(\sB_c; \Frob) \rightleftarrows HH(\sB_c^\ren; \Frob): HH(\unren; \Frob)\]
exhibiting $HH(\sB_c; \Frob)$ as a coreflective subcategory of $HH(\sB_c^\ren; \Frob)$. By \S\ref{sss:hh-A1-equiv} the functor $\pi = HH(\unren; \Frob)$ is $\A^1$-equivariant. By \Cref{prop:coact-Bren-str-cts} and \Cref{prop:hh-strong-cts-coact} the coaction functor for $HH(\sB_c^\ren;\Frob)$ is strongly continuous, so we may apply \Cref{prop:str-cts-semisimple-conservative} with $\sC = HH(\sB_c; \Frob)$, $\sD = HH(\sB_c^\ren; \Frob)$, and $\pi$ the functor between them. 
\end{proof}

\begin{proposition}\label{prop:hh-bc-equals-hh-jc}
\[HH(\sB_c; \Frob) \simeq HH(\sJ_c; \id)\]
\end{proposition}

\begin{proof}
By \Cref{prop:A1-hh-Bc-triv} and \Cref{prop:invts-hh} we know that 
\[HH(\sB_c; \Frob) \simeq HH(\sB_c; \Frob)^{\A^1} \simeq HH(\sB_c^{\A^1}; \Frob).\]
By \Cref{prop:Bc-A1-invts-Jc} we know that $\sB^{\A^1} \simeq \sJ_c$ (in fact, this is true as a monoidal category by \Cref{prop:monoidal-structure-A1-invts}), and by \Cref{prop:frob-action-Gm} we know that $\Frob$ acts trivially on $\sJ_c$, so the result follows. 
\end{proof}

\begin{theorem}\label{thm:hh-hc-jc}
\[HH(\sH_c; \Frob) \simeq HH(\sJ_c; \id).\]
\end{theorem}

\begin{proof}
This follows immediately from what we have already proved:
\[HH(\sH_c; \Frob) \overset{(\textrm{\Cref{prop:hh-hc-equals-hh-bc}})}{\simeq} HH(\sB_c; \Frob) \overset{(\textrm{\Cref{prop:hh-bc-equals-hh-jc}})}{\simeq} HH(\sJ_c; \id).\]
\end{proof}

\subsection{Proof of \Cref{thm:intro-main}}\label{ss:6-main}

\subsubsection{}\label{sss:vect-gamma-twist} If $\Gamma$ is a finite group and $k$ is an algebraically closed field of characteristic zero, then 
\[H^3(\Gamma; k^\times) \simeq H^3(\Gamma; \mu_\infty) \simeq H^4(\Gamma; \Z).\]
Given a cocycle $\omega \in H^3(\Gamma; k^\times)$, we get an isomorphism class of maps 
\[B\Gamma \to B^3\Gm.\]
Every map $B\Gamma \to B^3\Gm$ may be upgraded to a pointed map, and the space of pointed structures is connected (because we are working over an algebraically closed field). Applying $\Omega$, we obtain an $\E_1$-map $\Gamma \to B^2\Gm$, well-defined up to isomorphism. Pulling back the universal gerbe on $B^2\Gm$, we obtain a monoidal category $\Vect_{\Gamma, \omega}$. This category is isomorphic to $\QCoh(\Gamma)$ as a plain category, but the associativity constraint of its monoidal structure is modified according to $\omega$. 

\medskip 

The $(\infty, 2)$-category 
\[\Vect_{\Gamma, \omega}\tmod\]
is fully-dualizable, and corresponds under the cobordism hypothesis to a 3d TQFT which is called Dijkgraaf-Witten theory (with gauge group $\Gamma$ and level $\omega$). For more information, see \cite{quinn1998group}.

\begin{theorem}\label{thm:main}
The category $\Rep(G(\F_q))^\unip$ admits a direct sum decomposition indexed by two-sided cells 
\[\Rep(G(\F_q))^\unip \simeq \bigoplus_c \Rep(G(\F_q))^\unip_c\]
such that 
\[\Rep(G(\F_q))^\unip_c \simeq \QCoh_\sG(A_c/A_c)\]
for a finite group $A_c$ and gerbe $\sG$.  
\end{theorem}

\begin{proof}
By \S\ref{sss:trace-of-hecke}, 
\[\Rep(G(\F_q))^\unip \simeq HH(\sH; \Frob).\]
Therefore \Cref{constr:semiorthogonal-trace} implies that the system of semiorthogonal decompositions of $\sH$ defined in \S\ref{ss:defn-Hc} gives rise to a system of semiorthogonal decompositions of $\Rep(G(\F_q))^\unip$. Every semiorthogonal decomposition of a semisimple category is a direct sum decomposition, so we obtain 
\[\Rep(G(\F_q))^\unip \simeq \bigoplus_c \Rep(G(\F_q))^\unip_c\]
where 
\[\Rep(G(\F_q))^\unip_c = HH(\sH_c; \Frob).\]
Applying \Cref{thm:hh-hc-jc}, we know that
\[HH(\sH_c; \Frob) \simeq HH(\sJ_c; \id).\]
As explained in \S\ref{sss:intro-losev-ostrik}, \cite[Theorem 5.1]{Losev_Ostrik_2014}
implies the existence of an equivalence 
\[\sJ_c\tmod \simeq \Vect_{A_c, \omega}\]
for some finite group $A_c$ and some 3-cocycle $\omega$. Now the result follows from \cite{freed1994higher, willerton2008twisted}, which show that 
\[\tr(\id; \Vect_{A_c, \omega}) \simeq \QCoh_\sG(A_c/A_c)\]
for some gerbe $\sG$ defined in terms of $\omega$. 
\end{proof}

\subsubsection{}\label{sss:gerbe-trivalize} In fact, it is known that the gerbe $\sG$ appearing in \Cref{thm:main} may always be trivialized. This is because case-by-case calculations show that for all but three two-sided cells (one in type $E_7$ and two in type $E_8$), $\omega$ is trivial. These three cells with nontrivial $\omega$ are called exceptional, and for all three exceptional cells $A_c = \Z/2$. For the non-exceptional cells, the triviality of $\omega$ implies the triviality of $\sG$. On the other hand, for the exceptional cells, 
\[A_c/A_c \simeq B(\Z/2) \sqcup B(\Z/2)\]
and every gerbe on $B(\Z/2)$ is trivializable. 

\bibliographystyle{alpha}
\bibliography{text}

\end{document}